\documentclass[11pt]{amsart}
\usepackage{amscd,amssymb,stmaryrd}
\usepackage[all]{xy}

\title[Koszul duality and tilting equivalence]{On a common generalization of Koszul duality and tilting equivalence}
\author{Dag Madsen}

\address{Department of Mathematics, 215 Carnegie, Syracuse University,
Syracuse, NY 13244-1150, USA}
\email{dmadsen@syr.edu}

\newtheorem{lem}{Lemma}[subsection]
\newtheorem{prop}[lem]{Proposition}
\newtheorem{cor}[lem]{Corollary}
\newtheorem{thm}[lem]{Theorem}

\newtheorem*{cor2}{Corollary}
\theoremstyle{definition}
\newtheorem{defin}[lem]{Definition}
\newtheorem{example}[lem]{Example}

\newtheorem*{conj}{Conjecture}
\newtheorem*{ques}{Question}
\newtheorem*{defin2}{Definition}

\newcommand{\G}{\Gamma}
\newcommand{\K}{\operatorname{\mathsf{K}}}
\renewcommand{\L}{\Lambda}
\newcommand{\Z}{\mathbb Z}

\newcommand{\add}{\operatorname{\mathsf{add}}}
\newcommand{\cgr}{\operatorname{\mathsf{cgr}}}
\newcommand{\cK}{\operatorname{\mathsf{cK}}}
\newcommand{\cone}{\operatorname{cone}}
\newcommand{\End}{\operatorname{End}}
\newcommand{\Ext}{\operatorname{Ext}}
\newcommand{\gldim}{\operatorname{gldim}}
\newcommand{\Gr}{\operatorname{\mathsf{Gr}}}
\newcommand{\gr}{\operatorname{\mathsf{gr}}}
\newcommand{\Hom}{\operatorname{Hom}}
\newcommand{\id}{\operatorname{id}}
\newcommand{\lfd}{\operatorname{\mathsf{lfd}}}
\renewcommand{\mod}{\operatorname{\mathsf{mod}}}
\newcommand{\Mod}{\operatorname{\mathsf{Mod}}}
\newcommand{\op}{\operatorname{op}}
\newcommand{\pd}{\operatorname{pd}}
\newcommand{\soc}{\operatorname{soc}}
\renewcommand{\top}{\operatorname{top}}
\newcommand{\Tor}{\operatorname{Tor}}

\newcommand{\gsh}[1]{\langle #1 \rangle}
\newcommand{\ggen}[1]{\{ #1 \}}

\begin{document}

\maketitle

\begin{abstract}
We propose a new definition of Koszulity for graded algebras where the degree zero part has finite global dimension, but is not necessarily semi-simple. The standard Koszul duality theorems hold in this setting. We give an application to algebras arising from multiplicity free blocks of the BGG category $\mathcal O$.
\end{abstract}

\section*{Introduction}
Koszul algebras have traditionally been required to be graded algebras with semi-simple degree zero part, but Green, Reiten, and Solberg in \cite{TKos} defined a notion of Koszulity for more general graded algebras where the degree zero part is allowed to be arbitrary finite dimensional. The purpose of their work was to develop a unified approach to Koszul duality and tilting equivalence. Koszulity of an algebra $\L$ was defined with respect to a module $T$, leading to the name \emph{$T$-Koszul algebras}. In \cite{Ext}, we gave an account of the $T$-Koszul algebras of Green, Reiten, and Solberg from a derived category perspective. In this subject the right concepts seem to have been found, but in both papers mentioned the basic definitions have been a little too elaborate. We propose here a new definition of $T$-Koszul algebras that is more elementary and accessible.
\begin{defin2}[\ref{nydef}]
Let $\L=\bigoplus_{i\geq 0} \L_i$ be a graded algebra with $\gldim \L_0 < \infty$, and let $T$ be a graded $\L$-module concentrated in degree zero.
We say that $\L$ is \emph{$T$-Koszul} if both of the following conditions hold.
\begin{itemize}
\item[(i)] $T$ is a tilting $\L_0$-module.
\item[(ii)] $T$ is graded self-orthogonal as a $\L$-module.
\end{itemize}
\end{defin2}

We prove that the new definition is equivalent to the one given in \cite{Ext}, provided $\gldim \L_0 < \infty$ (Theorem \ref{comp}). As a consequence the three main theorems for Koszul algebras hold in this setting, namely, each $T$-Koszul algebra has a $T$-Koszul dual algebra which is the extension algebra of $T$ (Theorem \ref{kosdua}(a)), the $T$-Koszul dual of the $T$-Koszul dual algebra is isomorphic to the original algebra (Theorem \ref{kosdua}(b)), and there is a duality between categories of $T$-Koszul modules over $T$-Koszul dual algebras (Theorem \ref{duamod}). In subsection \ref{inftysec} we comment on the situation when $\gldim \L_0=\infty$.

New results in this paper include: a proof that that ungraded and graded extension algebras of $T$ are isomorphic even when $\L$ is allowed to be non-noetherian (Corollary \ref{iso}), a much simplified account of $T$-Koszul modules (Proposition \ref{fincog} and Theorem \ref{duamod}), a proof that under certain finiteness conditions $T$-Koszul dual algebras have equivalent bounded categories (Theorem \ref{derequiv}), and a new description of so-called $\Delta$-filtered modules over certain quasi-hereditary Koszul algebras (Corollary \ref{gqh}).

We present an application to the representation theory of semi-simple Lie algebras that might be of particular interest. There is a $T$-Koszul duality acting behind the scenes in the paper \cite{Stan}, and it can be given the following surprisingly simple formulation.

\begin{cor2}[to Theorem \ref{hoved}]
Let $A_\lambda$ be a basic algebra corresponding to a block $\mathcal O_\lambda$ of category $\mathcal O$ for a semi-simple Lie algebra.  Let $\Delta$ denote the direct sum of all Verma modules in $\mathcal O_\lambda$. Let $E_\lambda=[\bigoplus_{i \geq 0} \Ext_{A_\lambda}^i(\Delta,\Delta)]^{\op}$. Suppose $A_\lambda$ is multiplicity free. Then $A_\lambda \cong [\bigoplus_{i \geq 0} \Ext_{E_\lambda}^i(D \Delta,D \Delta)]^{\op}$ as ungraded algebras.
Furthermore, if $A_\lambda$ and $E_\lambda$ are given the $\Ext$-grading, then there is an equivalence of triangulated categories $\mathcal D^b (\gr A_\lambda) \rightarrow \mathcal D^b (\gr E_\lambda)$.
\end{cor2}

In section 1 we recall some basic notions and facts about graded algebras and modules. In section 2 we discuss tilting equivalence and Koszul duality, highlighting the similarities between the two theories. In section 3 we look at graded self-orthogonal modules and the functors they define between derived categories. In section 4 we discuss the new definition of $T$-Koszul algebras and prove all the Koszul duality theorems. In section 5 we give an application to certain quasi-hereditary Koszul algebras, in particular multiplicity free blocks of category $\mathcal O$.

\section{Preliminaries on graded algebras and modules}

We recall some basic notions and facts regarding graded algebras and modules that we will use throughout the paper.

\subsection{Graded duality}
Let $k$ be a field and let $\L=\bigoplus_{i\geq 0} \L_i$ be a graded $k$-algebra. Throughout the paper we assume $\L$ is \emph{locally finite dimensional}, that is to say, we have $\dim_k \L_i < \infty$ for all $i \geq 0$. It is important to note that for most of the paper we do not assume $\L_0$ is semi-simple or commutative. To simplify statements of results, we will however assume throughout that $\L_0$ is augmented over $k \times \cdots \times k=k^{\times r}$ for some $r>0$, in particular we assume that $\L_0$ is a basic $k$-algebra.

We denote by $\Gr \L$ the category of graded $\L$-modules $M=\bigoplus_{i \in
\Z} M_i$ with degree $0$ morphisms. We denote by $\lfd \L$ the full
subcategory of locally finite dimensional modules, that is, graded $\L$-modules with $\dim_k M_i < \infty$ for all $i \in \Z$. Important full subcategories of $\lfd \L$ are the category of finitely generated graded $\L$-modules $\gr \L$ and the category of finitely cogenerated graded $\L$-modules $\cgr \L$.

Forgetting the grading of $\L$, we denote by $\Mod \L$ the category of (ungraded) $\L$-modules. If $A$ is a finite dimensional $k$-algebra, then we denote by $\mod A$ the category of finitely generated (ungraded) $A$-modules.

If $V=\bigoplus_{i \in \Z} V_i$ is a graded $k$-vector space, then its \emph{graded dual} vector space $DV$ is given by $(DV)_i=\Hom_k(V_{-i},k)$. If $M=\bigoplus_{i \in \Z} M_i$ is a graded $\L$-module, then $DM=\bigoplus_{i \in \Z} \Hom_k(M_{-i},k) $ can be given a graded $\L^{\op}$-module structure by $\lambda_j f_i(m)=f_i(\lambda_j m)$ for $\lambda_j \in \L_j$, $f_i \in \Hom_k(M_{-i},k)$ and $m \in M_{-i-j}$. If $M$ is locally finite dimensional, then $DDM \simeq M$.

The duality $D$ can also in a straightforward way be defined for morphisms of graded modules, and we get a duality functor $D \colon \lfd \L \rightarrow \lfd \L^{\op}$. This functor restricts to dualities $D \colon \gr \L \rightarrow \cgr \L^{\op}$ and $D \colon \cgr \L \rightarrow \gr \L^{\op}$. Restricting further to finite dimensional modules concentrated in degree zero, we get a duality $D \colon \mod \L_0 \rightarrow \mod \L^{\op}_0$. If $\L=\L_0$, then this is the usual duality $D \colon \mod \L \rightarrow \mod \L^{\op}$ for finite dimensional algebras.

Given a graded $\L$-module $M$, the \emph{$j$th graded shift} of M, denoted $M \gsh j$, is the module with graded parts $(M \gsh j)_i = M_{i-j}$ and module structure inherited from $M$. There is a functorial isomorphism $D(M \gsh j) \simeq DM \gsh {-j}$.

\subsection{Graded extensions}

The category $\Gr \L$ is abelian, and one can define the usual homological functors. The following well-known result can for instance be found in \cite[Corollaries 2.4.4 and 2.4.7]{Grad}.

\begin{lem}\label{ugrad}
Let $\L=\bigoplus_{i \geq 0} \L_i$ be a graded algebra, and let $M$ and
$N$ be graded $\L$-modules.
\begin{itemize}
\item[(a)]If $M$ is finitely generated, then $$\Hom _{\L}(M,N) \simeq
\bigoplus_{j \in \Z} \Hom _{\Gr \L}(M,N \gsh j)$$
\item[(b)]If $M$ is finitely generated and has a
projective resolution such that all syzygies are finitely generated, then $$\Ext^i_{\L}(M,N) \simeq \bigoplus_{j \in
\Z} \Ext^i_{\Gr \L}(M,N \gsh j)$$ for all $i \geq 0$.
\end{itemize}
\end{lem}

With some finiteness conditions, the duality $D$ and graded extensions combine well. A reference for the following proposition is \cite[Proposition 1.3]{Ext}.

\begin{prop}
Let $\L=\bigoplus_{i \geq 0} \L_i$ be a graded algebra, let $M$ and
$N$ be graded locally finite dimensional $\L$-modules, and let $D$ denote the
above duality. If $M$ is bounded below or $N$ is bounded above, then
$$\Ext^i_{\Gr \L}(M,N) \simeq \Ext^i_{\Gr \L^{\op}}(DN,DM)$$
for all $i \geq 0$.
\end{prop}

\section{Tilting equivalence and Koszul duality}
In this section we present the basic facts about tilting equivalence and Koszul duality in such a way that the similarities between the two theories are made apparent.

Let $T$ be an object in an abelian category $\mathcal A$. Then $\add T$ denotes the full subcategory of $\mathcal A$ consisting of all finite direct sums of direct summands of $T$. The derived category of $\mathcal A$ is denoted by $\mathcal D (\mathcal A)$. The \emph{bounded derived category} of $\mathcal A$, that is, the full subcategory of $\mathcal D (\mathcal A)$ consisting of objects with bounded homology, is denoted by $\mathcal D^b(\mathcal A)$.

\subsection{Tilting and cotilting modules}

\begin{defin}
Let $A$ be a basic finite dimensional $k$-algebra and let $T$ be a finitely generated $A$-module. The module $T$ is called a \emph{tilting module} if all of the following conditions hold.
\begin{itemize}
\item[(i)] $\pd T < \infty$.
\item[(ii)] $\Ext_A^i(T,T)=0$ for all $i>0$.
\item[(iii)] There exists a coresolution $0 \to A \to T_0 \to T_1 \to \ldots \to T_n \to 0$ with $T_i$ in $\add T$ for all $0 \leq i \leq n$.
\end{itemize}
\end{defin}

There is a dual notion of cotilting module. A finitely generated $A$-module $U$ is called a \emph{cotilting module} if (i) $\id U < \infty$, (ii) $\Ext_A^i(U,U)=0$ for all $i>0$ and (iii) there exists a resolution $0 \to U_m \to \ldots \to U_1 \to U_0 \to DA \to 0$ with $U_i$ in $\add U$ for all $0 \leq i \leq m$.

There is a more general concept due to Wakamatsu \cite{Waka}. A finitely generated $A$-module $T$ is called a \emph{Wakamatsu tilting module} if (i) $\Ext_A^i(T,T)=0$ for all $i>0$ and (ii) there exists a coresolution $$0 \rightarrow A \rightarrow T_0 \xrightarrow {f_0} T_1 \xrightarrow {f_1} T_2 \rightarrow \ldots \rightarrow T_n \rightarrow \ldots$$ with $T_i$ in $\add T$ for all $i \geq 0$ and $\Ext^1_A(\ker f_i,T)=0$ for $i>0$. Wakamatsu cotilting modules are defined dually. It is known that a module is a Wakamatsu tilting module if and only if it is a Wakamatsu cotilting module \cite{Cot}.
Tilting modules and cotilting modules are Wakamatsu tilting modules, but in general a Wakamatsu tilting module can have both infinite projective dimension and infinite injective dimension. There is the following important conjecture about Wakamatsu tilting modules with finite projective dimension.

\begin{conj}[Wakamatsu Tilting Conjecture \cite{Apo}]
Let $T$ be a Wakamatsu tilting module with finite projective dimension. Then $T$ is a tilting module.
\end{conj}

If $A$ has finite global dimension, then infinite exact sequences will eventually split, and we have the following consequence.

\begin{prop}\label{fin}
Let $A$ be a finite dimensional algebra with $\gldim A < \infty$. Let $T$ be a finitely generated $A$-module. The following are equivalent.
\begin{itemize}
\item[(a)] $T$ is a tilting module.
\item[(b)] $T$ is a cotilting module.
\item[(c)] $T$ is a Wakamatsu tilting module.
\end{itemize}
\end{prop}

For the rest of the paper we will assume that all tilting, cotilting and Wakamatsu tilting module we deal with are \emph{basic}, that is, the indecomposable direct summands of a module are non-isomorphic. If $A$ is augmented over $k^{\times r}$, then basic tilting or cotilting modules have $r$ indecomposable direct summands. It is not known whether the corresponding statement holds for Wakamatsu tilting modules.

\subsection{Tilting equivalence}

Let $A$ be a basic finite dimensional $k$-algebra and suppose $T$ is a tilting $A$-module. Let $B=[\End_A(T)]^{\op}$. The $B$-module $DT \simeq \Hom_A(T,DA)$ is a cotilting module, and there is a ring isomorphism $A \cong [\End_B(DT)]^{\op}$. According to the Morita theory for derived categories \cite{Rick} \cite[chapter 6]{Grp}, there exists a two-sided tilting complex $X$ and an
equivalence of triangulated categories $$\bar G_T=\mathbb R \Hom_A(X,-)
\colon \mathcal D^b(\mod A) \rightarrow \mathcal D^b(\mod B)$$ such that $H^i \bar G_T (M)=\Ext^i_A(T,M)$ for any $A$-module $M$ and $i \geq 0$. A quasi-inverse of $\bar G_T$ is given by $\bar F_T= X \otimes_{B}^{\mathbb L} - \colon \mathcal D^b(\mod B) \rightarrow \mathcal D^b(\mod A)$. The functor $\bar F_T$ has the property that $H^{-n}\bar F_T(N)=\Tor_n^B(T,N)$ for any $B$-module $N$ and $n \geq 0$.

The equivalence above restricts to an equivalence between certain
module categories. Define $$T^{\perp}=\{M \in \mod A \mid \Ext^i_A(T,M)=0, \; i>0 \}$$ and $${}^{\perp}DT=\{N \in \mod B \mid \Ext^i_B(N,DT)=0, \; i>0 \}.$$ We get an equivalence $\bar E_T=\Hom_A(T,-) \colon T^{\perp} \rightarrow {}^{\perp}DT$ with quasi-inverse $T \otimes_B -$. $$\xymatrix @M=1.5ex {\mathcal D^b(\mod A) \ar[rr]^{\mathbb R \Hom_A(X,-)}  && \mathcal D^b(\mod B)\\
T^{\perp} \ar[rr]^{\Hom_A(T,-)} \ar@{^{(}->}[u] && {}^{\perp}DT
\ar@{^{(}->}[u]}$$

Important objects in $T^{\perp}$ are $T$ itself and the injective cogenerator $DA$. Important objects in ${}^{\perp}DT$ are $DT$ and $B$. At the core of the equivalence $\bar E_T$ there is a correspondence between indecomposable summands of $T$ and indecomposable projective $B$-modules and another correspondence between indecomposable injective $A$-modules and indecomposable summands of $DT$.
\begin{eqnarray*}
&T \overset {\bar E_T} {\longmapsto}  B, \qquad &DA \overset {\bar E_T} {\longmapsto}  DT,\\&T \underset {\bar F_T} {\longmapsfrom} B, \qquad &DA \underset {\bar F_T} {\longmapsfrom} DT.
\end{eqnarray*}

If we wish to make the correspondence more symmetrical, we can compose $\bar E_T$ with the duality $D$ and get a duality $\bar E_T D=\Hom_{A^{\op}}(-,DT) \colon {}^{\perp}{}_{A^{\op}}DT \rightarrow {}^{\perp}{}_B DT$.
$$\xymatrix @M=1.5ex {& \mathcal D^b(\mod A) \ar[r]^{\bar G_T}  & \mathcal D^b(\mod B)\\ {}^{\perp}_{A^{\op}}DT \ar[r]^D &
{}_A T^{\perp} \ar[r]^{\bar E_T} \ar@{^{(}->}[u] & {}^{\perp}_B DT
\ar@{^{(}->}[u]}$$ In this way indecomposable direct summands of $DT$ on one side correspond to indecomposable projective modules on the other side in both directions.
\begin{eqnarray*}
&DT \stackrel {\bar E_T D} {\longmapsto} B, \qquad &A \overset {\bar E_T D} {\longmapsto}  DT,\\&DT \underset {D \bar F_T} {\longmapsfrom} B, \qquad &A \underset {D \bar F_T} {\longmapsfrom} DT.
\end{eqnarray*}

\subsection{Koszul duality}

Koszul algebras were first defined in \cite{Pri}. We follow here the treatment in \cite{Bei} in which algebras are supposed to be graded with semi-simple degree zero part.

\begin{defin}
Let $\L=\bigoplus_{i\geq 0} \L_i$ be a graded algebra where $\dim_k\L_i < \infty$ for all $i \geq 0$ and $\L_0 \cong k \times \cdots \times k=k^{\times r}$ as rings. Then $\L$ is called a \emph{Koszul algebra} if $\Ext^i_{\Gr \L}(\L_0,\L_0 \gsh j)=0$ whenever $i \neq j$.
\end{defin}

Suppose $\L$ is a Koszul algebra. By \cite[Proposition 2.1.3]{Bei}, this is equivalent to $\L_0$ having a linear projective $\L$-resolution. In particular $\L_0$ has a projective $\L$-resolution such that all syzygies are finitely generated. Let $$\G=[\bigoplus_{i\geq 0} \Ext_\L^i (\L_0,\L_0)]^{\op} \cong [\bigoplus_{i\geq 0} \Ext^i_{\Gr \L}(\L_0,\L_0 \gsh i)]^{\op}.$$ Then $\G$ is also a Koszul algebra \cite[Theorem 2.10.2]{Bei}, it is called the \emph{Koszul dual} of $\L$. There is an isomorphism of graded algebras $$\L \cong [\bigoplus_{i\geq 0} \Ext_\G^i (\G_0,\G_0)]^{\op} \cong [\bigoplus_{i\geq 0} \Ext^i_{\Gr \G}(\G_0,\G_0 \gsh i)]^{\op}.$$ According to the theory of lifts as laid out in \cite[7.3 and 10.2]{Dgc}, by lifting the set $\{\L_0 \gsh i [i] \}_{i \in \Z}$, we can construct a bigraded $\L$-$\G$-bimodule complex $X$ and a triangulated functor $$G=\mathbb R \Hom_{\Gr \L}(X,-) \colon \mathcal D(\Gr \L) \rightarrow \mathcal D(\Gr \G)$$ with the property that $(H^i G (M))_j=\Ext^{i+j}_{\Gr\L}(\L_0,M \gsh j)$ for any graded $\L$-module $M$ and $i,j \in \Z$. We recall the steps of this construction in section \ref{lift}. A left adjoint of $G$ is given by $F= X \otimes_{\Gr \G}^{\mathbb L} - \colon \mathcal D(\Gr \G) \rightarrow \mathcal D(\Gr \L)$.

The functor $G$ restricts to an equivalence between certain categories of graded modules. Define the category of \emph{coKoszul $\L$-modules} $$\cK(\L)=\{M \in \cgr \L \mid \Ext^i_{\Gr \L}(\L_0 \gsh{-j},M)=0 \text{ whenever } i \neq j \}$$ and the category of \emph{Koszul $\G$-modules} $$\K(\G)=\{N \in \gr \G \mid \Ext^i_{\Gr \G}(N,\G_0 \gsh j)=0 \text{ whenever } i \neq j \}.$$ We get an equivalence $E \colon \cK(\L) \rightarrow \K(\G)$, where $E=\bigoplus_{i\geq 0} \Ext^i_{\Gr \L}(\L_0 \gsh{-i},-) \simeq \bigoplus_{i\geq 0} \Ext^i_\L(\L_0,-)$.
$$\xymatrix @M=1.5ex {\mathcal D(\Gr \L) \ar[rr]^{\mathbb R \Hom_{\Gr \L}(X,-)}  && \mathcal D(\Gr \G)\\
\cK(\L) \ar[rr]^E \ar@{^{(}->}[u] && \K(\G)
\ar@{^{(}->}[u]}$$

Important objects in $\cK(\L)$ are $\L_0$ and the graded injective cogenerator $D\L$. Important objects in $\K(\G)$ are $\G_0$ and $\G$. At the core of the equivalence $E$ there is a correspondence between indecomposable graded simple $\L$-modules and indecomposable projective $\G$-modules and another correspondence between indecomposable graded injective $\L$-modules and graded simple $\G$-modules.
\begin{eqnarray*}
&\L_0 \overset {E} {\longmapsto}  \G, \qquad &D\L \overset {E} {\longmapsto} \G_0,\\&\L_0 \underset {F} {\longmapsfrom} \G, \qquad &D\L \underset {F} {\longmapsfrom} \G_0.
\end{eqnarray*}

If we wish to make the correspondence more symmetrical, we can compose $E$ with the duality $D$ and get a duality $ED \simeq \bigoplus_{i\geq 0} \Ext^i_{\L^{\op}}(-,\L_0) \colon \K(\L^{\op}) \rightarrow \K(\G)$.
$$\xymatrix @M=1.5ex {& \mathcal D(\Gr \L) \ar[r]^G  & \mathcal D(\Gr \G)\\ \K(\L^{\op}) \ar[r]^D &
\cK(\L) \ar[r]^E \ar@{^{(}->}[u] & \K(\G)
\ar@{^{(}->}[u]}$$ In this way graded simple modules on one side correspond to indecomposable projective modules on the other side in both directions.
\begin{eqnarray*}
&\L_0 \overset {ED} {\longmapsto}  \G, \qquad &\L \overset {ED} {\longmapsto} \G_0,\\&\L_0 \underset {DF} {\longmapsfrom} \G, \qquad &\L \underset {DF} {\longmapsfrom} \G_0.
\end{eqnarray*}

We often refer to the Koszul algebras of this section as \emph{classical} Koszul algebras.

\section{Graded self-orthogonal modules}

We now return to graded algebras $\L=\bigoplus_{i\geq 0} \L_i$, where $\L_0$ is a finite dimensional algebra augmented over $k^{\times r}$, but not necessarily semi-simple. Let $1=e_1 + \ldots + e_r$ be the decomposition of the identity in primitive orthogonal idempotents. Our aim is to describe a duality theory encompassing both tilting equivalence and Koszul duality. In this section we look at suitable functors for such a duality theory.

\subsection{Graded and ungraded extension algebras}

Motivated by the definition of classical Koszul algebras, we define the following homological property for graded modules.

\begin{defin}
Let $T$ be a finitely generated basic graded $\L$-module concentrated in degree $0$. We say that $T$ is a \emph{graded self-orthogonal module} if $\Ext^i_{\Gr \L}(T,T \gsh j)=0$ whenever $i \neq j$.
\end{defin}

Examples of graded self-orthogonal modules are the tilting modules in the $\L=\L_0$ case and the module $\L_0$ in the classical Koszul case.

In general such a module $T$ might have syzygies that are not finitely generated. Even though we are not in a situation where Lemma \ref{ugrad} (b) applies, we can establish the following isomorphisms.

\begin{prop}\label{ungr}
Let $T$ be a graded self-orthogonal module. Then $$\Ext^i _{\L}(T,T) \simeq \Ext^i _{\Gr \L}(T,T \gsh i)$$ for each $i \geq 0$.
\end{prop}

\begin{proof} Suppose $T$ has a minimal graded projective resolution $$\mathcal P \colon \ldots \to P^{-i} \to \dots \to P^{-1} \to P^0$$ For each $i>0$ it might happen that $P^{-i}$ is not finitely generated, but, since $T$ is bounded below, we always have $P^{-i}=\bigoplus_{l \in \Z} P^{-i} \ggen l$, where $P^{-i} \ggen l$ is a graded projective $\L$-module finitely generated in degree $l$, for each $l \in \Z$. By forgetting the grading, we can regard $\mathcal P$ as an ungraded projective resolution of $T$. The complex $\Hom_\L(\mathcal P, T)$ computes the groups $\Ext^i _{\L}(T,T)$. The $i$th term of this complex is $\Hom_{\L}(\bigoplus_{l \in \Z} P^{-i} \ggen l,T)$. For each $i \geq 0$ we have \begin{eqnarray*}
\Hom_{\L}(\bigoplus_{l \in \Z} P^{-i} \ggen l,T) &\simeq& \prod_{l \in \Z} \Hom_{\L}(P^{-i} \ggen l,T)\\ &\simeq& \prod_{l \in \Z} \Hom_{\Gr \L}(P^{-i} \ggen l,T \gsh l),
\end{eqnarray*}
since $T$ is concentrated in degree $0$. Also, since $T$ is concentrated in a single degree, we have $$d(\Hom_{\Gr \L}(P^{-i} \ggen l,T \gsh l)) \subseteq \Hom_{\Gr \L}(P^{-(i-1)} \ggen l,T \gsh l)$$ for all $i>0$, $l \in \Z$, where $d$ is the differential in $\Hom_\L(\mathcal P, T)$. The complex $\Hom_\L(\mathcal P, T)$ is then an infinite product (indexed by $l$) of complexes $\Hom_\L(\mathcal P \ggen l, T \gsh l)$ of the form
\begin{multline*}
\ldots \rightarrow \Hom_{\Gr \L}(P^{-i} \ggen l,T \gsh l) \rightarrow \Hom_{\Gr \L}(P^{-(i-1)} \ggen l,T \gsh l) \rightarrow \ldots\\ \ldots \rightarrow \Hom_{\Gr \L}(P^{-1} \ggen l,T \gsh l) \rightarrow \Hom_{\Gr \L}(P^0 \ggen l,T \gsh l) \rightarrow 0.
\end{multline*}
For each $l \in \Z$, since $T \gsh l$ is concentrated in degree $l$, the complex $\Hom_\L(\mathcal P \ggen l, T \gsh l)$ computes the groups $\Ext^i _{\Gr \L}(T,T \gsh l)$. We conclude that $$\Ext^i_{\L}(T,T) \simeq \prod_{l \in \Z} \Ext^i_{\Gr \L}(T,T \gsh l)$$ for each $i \geq 0$. But $\Ext^i_{\Gr \L}(T,T \gsh l)=0$ whenever $i \neq l$, so $$\Ext^i_{\L}(T,T) \simeq \Ext^i_{\Gr \L}(T,T \gsh i)$$ for each $i \geq 0$.
\end{proof}

We conclude that the graded and ungraded extension algebras are isomorphic.

\begin{cor}\label{iso}
Let $T$ be a graded self-orthogonal module. There is an isomorphism of graded algebras $$[\bigoplus_{i\geq 0} \Ext_\L^i (T,T)]^{\op} \cong [\bigoplus_{i\geq 0} \Ext^i_{\Gr \L}(T,T \gsh i)]^{\op}.$$
\end{cor}

\begin{proof}
From Proposition \ref{ungr} we know there is a bijective linear map between the algebras. To prove that multiplication is preserved, we codify each algebra as a category and show that the resulting categories are equivalent. (See \cite[section 4]{Ext} for the details on how to codify a graded algebra as a category.) The graded algebra $[\bigoplus_{i\geq 0} \Ext_\L^i (T,T)]^{\op}$ can be codified as the full subcategory $\mathcal U=\{T [i] \mid i \in \Z\} \subseteq \mathcal D (\Mod \L)$.  The graded algebra $Ext^i_{\Gr \L}(T,T \gsh i)]^{\op}$ can be codified as the full subcategory $\mathcal V=\{T \gsh i [i] \mid i \in \Z\} \subseteq \mathcal D (\Gr \L)$. There is a forgetful functor $|-| \colon D (\Gr \L) \to D (\Mod \L)$ which by (the proof of) Proposition \ref{ungr} restrict to an equivalence $|-| \colon \mathcal V \to \mathcal U$. Therefore $[\bigoplus_{i\geq 0} \Ext_\L^i (T,T)]^{\op}$ and $[\bigoplus_{i\geq 0} \Ext^i_{\Gr \L}(T,T \gsh i)]^{\op}$ are isomorphic as graded algebras.
\end{proof}

\subsection{Derived functors}\label{lift}
Let $T=\bigoplus_{h=1}^r T_{(h)}$ be a graded self-orthogonal module, where each direct summand $T_{(h)}$ is indecomposable, and let $$\mathcal P \colon \ldots \to P^{-i} \to \ldots \to P^{-1} \to P^0$$ be a graded projective resolution of $T$. Consider the class of objects $$\mathcal U =\{T_{(h)} \gsh i [i] \mid 1 \leq h \leq r , i \in \Z \}$$ in $\mathcal D (\Gr \L)$. Constructing the \emph{standard lift} \cite[7.3]{Dgc} of $\mathcal U$, we get the graded DG algebra $\tilde \G$ with components $$\tilde \G^j_i=\bigoplus_{l \in \Z}\Hom_{\Gr \L}(P^l,P^{l+i+j}\gsh i)$$ and differential $$f \mapsto d \circ f - (-1)^n f \circ d$$ for $f \in \tilde \G^j$, and a bigraded DG $\L$-$\tilde \G$-bimodule $\tilde X$ with components $$\tilde X^l_{i,j}=P^{l+j}_{i-j}.$$ Let $\mathcal D (\tilde \G)$ denote the derived category of $\tilde \G$. There is a pair of adjoint triangulated functors $$\tilde G_T=\mathbb R \Hom_{\Gr \L}(\tilde X,-) \colon \mathcal D(\Gr \L) \rightarrow \mathcal D(\tilde \G)$$ and $$\tilde F_T= \tilde X \otimes^{\mathbb L}_{\tilde \G} - \colon \mathcal D(\tilde \G) \rightarrow \mathcal D(\Gr \L).$$ By the periodicity of $\tilde X$, we have $\tilde G_T(N \gsh i)= \tilde G_T (N) \gsh {-i} [-i]$ for any $N \in \mathcal D(\Gr \L)$. Since $T$ is graded self-orthogonal, the DG algebra $\tilde \G$ is formal and has cohomology algebra $$\G=[\bigoplus_{i\geq 0} \Ext^i_{\Gr \L}(T,T \gsh i)]^{\op} \cong [\bigoplus_{i\geq 0} \Ext_\L^i (T,T)]^{\op}.$$ By \cite[9.1]{Dgc}, we can find a bigraded DG $\tilde \G$-$\G$-bimodule $Y$ such that $$H_{\tilde \G}=\mathbb R \Hom_{\tilde \G}(Y,-) \colon \mathcal D (\tilde \G) \to \mathcal D(\Gr \G)$$ is an equivalence of triangulated categories. This equivalence commutes with the graded shift $\gsh 1$. Now let $G_T=H_{\tilde \G} \tilde G_T$. By \cite[6.3(c)]{Dgc}, there is a complex of bigraded $\L$-$\G$-modules $X$ such that $$G_T=\mathbb R \Hom_{\Gr \L}(X,-) \colon \mathcal D (\Gr \L) \to \mathcal D(\Gr \G).$$ This functor has a left adjoint $$F_T= X \otimes^{\mathbb L}_{\Gr \G} - \colon \mathcal D(\Gr \G) \rightarrow \mathcal D(\Gr \L).$$ We let $\phi \colon F_T G_T \rightarrow \id_{\mathcal D (\Gr \L)}$ denote the counit of the adjunction.

It is perhaps fortunate that we won't be needing the actual construction of the complex $X$ and the functor $G_T$. Besides being triangulated, the other important facts about $G_T$ can be summarized in the following proposition.

\begin{prop}\label{rhom}
Let $T$ be a graded self-orthogonal $\L$-module. Let $M$ be a finitely cogenerated $\L$-module, let $N$ be an object in $\mathcal D(\Gr \L)$, and let $i,j \in \Z$. Then
\begin{itemize}
\item[(a)] $G_T(T) \simeq \G$.
\item[(b)] The map $\phi_T \colon F_T G_T(T) \rightarrow T$ is an isomorphism.
\item[(c)] There is a functorial isomorphism $G_T(N \gsh j) \simeq G_T(N) \gsh {-j} [-j].$
\item[(d)] There is a functorial isomorphism $F_T G_T(N \gsh j) \simeq F_T G_T(N) \gsh j$.
\item[(e)] $(H^i G_T(M))_j \simeq \Ext^{i+j}_{\Gr \L}(T,M \gsh j).$
\item[(f)] $G_T(D\L) \simeq DT$, where $DT \simeq \Hom_{\Gr \L}(T,D\L)$ is considered as a graded $\G$-module concentrated in degree $0$.
\end{itemize}
\end{prop}

\begin{proof}
(a) and (b) One of the defining properties of a lift is that $\tilde G_T$ restricts to an equivalence between $\mathcal U$ and the category consisting of the indecomposable direct summands of the DG modules $\tilde \G \gsh i$, see \cite[7.3]{Dgc}. Composing with the equivalence $H_{\tilde \G}$, we get assertions (a) and (b).

(c) By the definition of $G_T$ it follows that
\begin{eqnarray*}
G_T(N \gsh j) &=& H_{\tilde \G} \tilde G_T(N \gsh j)\\
&=& H_{\tilde \G} (\tilde G_T(N) \gsh {-j} [-j])\\
&\simeq& (H_{\tilde \G} \tilde G_T(N)) \gsh {-j} [-j]\\
&=& G_T(N) \gsh {-j} [-j].
\end{eqnarray*}

(d) For any object $L$ in $\mathcal D(\Gr \L)$, there are isomorphisms
\begin{eqnarray*}
\Hom_{\mathcal D \Gr \L}(F_T G_T (N \gsh j),L)&\simeq& \Hom_{\mathcal D \Gr \G}(G_T (N \gsh j),G_T(L))\\ &\simeq& \Hom_{\mathcal D \Gr \G}(G_T(N) \gsh {-j}[-j],G_T(L))\\ &\simeq& \Hom_{\mathcal D \Gr \G}(G_T(N),G_T(L) \gsh j [j])\\ &\simeq& \Hom_{\mathcal D \Gr \L}(F_T G_T (N),L \gsh {-j})\\ &\simeq& \Hom_{\mathcal D \Gr \L}(F_T G_T (N) \gsh j,L).
\end{eqnarray*}

(e) There are isomorphisms
\begin{eqnarray*}
(H^i G_T(M))_j &\simeq& \Hom_{\mathcal D \Gr \G}(\G \gsh j,G_T(M)[i])\\ &\simeq& \Hom_{\mathcal D \Gr \L}(T \gsh {-j}[-j],M[i])\\
&\simeq& \Hom_{\mathcal D \Gr \L}(T, M \gsh j [i+j])\\
&\simeq& \Ext^{i+j}_{\Gr \L}(T,M \gsh j).
\end{eqnarray*}

(f) Since $D\L$ is injective, the group $(H^i G_T(D\L))_j \simeq \Ext^{i+j}_{\Gr \L}(T,D\L \gsh j)$ is non-zero only when $i+j=0$. Since $T$ is concentrated in degree $0$, we have $\Hom_{\Gr \L}(T,D\L \gsh j) \neq 0$ only when $j=0$ and as a consequence $i=0$. Therefore, $G_T(D\L) \simeq (H^0 G_T(D\L))_0 \simeq \Hom_{\Gr \L}(T,D\L) \simeq DT.$
\end{proof}

The question of whether $\phi_{D\L} \colon F_T G_T(D\L) \rightarrow D\L$ is an isomorphism is a central one, and we will return to it in the next section.

\section{$T$-Koszul algebras}

$T$-Koszul algebras were first defined by Green, Reiten, and Solberg \cite{TKos} with the purpose of developing a unified approach to Koszul duality and tilting equivalence. The original definition is rather involved, and we will not repeat the details here. In \cite{Ext} we observed that the two central properties of $T$-Koszul algebras are (i) $T$ is graded self-orthogonal as a $\L$-module, and (ii) $\phi_{D\L} \colon F_T G_T(D\L) \rightarrow D\L$ is an isomorphism. These two conditions taken together are strong enough to ensure the existence of a good duality theory. This fact led us in \cite{Ext} to make a simplified definition and say that $\L$ is a  $T$-Koszul algebra if and only if (i) and (ii) hold.

The condition on $\phi_{D\L}$, while bringing much of the power, might not be so tractable, and we present here a new equivalent definition where that condition is replaced with a more elementary one. The only slightly unfortunate aspect of the new definition is that we have to assume $\gldim \L_0 < \infty$, but most of the examples we are interested in are still covered.

\subsection{The new definition}

\begin{defin}\label{nydef}
Let $\L=\bigoplus_{i\geq 0} \L_i$ be a graded algebra with $\gldim \L_0 < \infty$, and let $T$ be a graded $\L$-module concentrated in degree $0$.
We say that the algebra $\L$ is \emph{$T$-Koszul} (or, \emph{Koszul with respect to $T$}) if both of the following conditions hold.
\begin{itemize}
\item[(i)] $T$ is a tilting $\L_0$-module.
\item[(ii)] $T$ is graded self-orthogonal as a $\L$-module.
\end{itemize}
\end{defin}

The tilting and classical Koszul algebra cases that we want to generalize can be recovered by setting $\L=\L_0$ and $\L_0 \cong k^{\times r}$ respectively.

If $\L=\L_0$, then (i) implies (ii), so $\L$ is Koszul with respect to $T$ if and only if $T$ is a tilting module.

If $\L_0 \cong k^{\times r}$, then the only tilting $\L_0$-module is $\L_0$ itself. In this case $\L$ is Koszul with respect to $T=\L_0$ if and only $\L$ is a classical Koszul algebra.

A great benefit of the new definition of $T$-Koszul algebras is that it is easy to establish the following important theorem. It was also obtained in \cite{Mag} with the original definition.

\begin{thm}\label{opp}
Let $\L=\bigoplus_{i\geq 0} \L_i$ be a graded algebra with $\gldim \L_0 < \infty$. Suppose $\L$ is a Koszul algebra with respect to a module $T$. Then $\L^{\op}$ is a Koszul algebra with respect to $DT$.
\end{thm}

\begin{proof}
If $T$ is a tilting $\L_0$-module, then $DT$ is a cotilting $\L_0^{\op}$-module. Since $\gldim \L_0^{\op}=\gldim \L_0$ is finite, any cotilting $\L_0^{\op}$-module is also a tilting $\L_0^{\op}$-module.

Using the duality $D$ we get an isomorphism $$\Ext^i_{\Gr \L}(T,T \gsh j) \simeq \Ext^i_{\Gr \L^{\op}}(DT \gsh {-j},DT) \simeq \Ext^i_{\Gr \L^{\op}}(DT,DT \gsh j).$$ So $T$ is graded self-orthogonal as a $\L$-module if and only if $DT$ is graded self-orthogonal as a $\L^{\op}$-module.
\end{proof}

If $\L$ is a Koszul algebra with respect to $\L_0$, then $\L^{\op}$ is a Koszul algebra with respect to $D \L_0$. As the following example shows, it does not necessarily follow that $\L$ is a Koszul algebra with respect to $D \L_0$.

\begin{example}
Let $\L$ be the path algebra given by the quiver
$$\xymatrix{1 \ar[r]^\alpha & 2 \ar[r]^\beta & 3}$$
and relation $\beta \alpha=0$. Let $\deg \alpha =2$ and $\deg \beta=0$.

Then $\L$ is a Koszul algebra with respect to $\L_0=S_1 \oplus P_2 \oplus S_3$. On the other hand $\L$ is not a Koszul algebra with respect to $D\L_0=S_1 \oplus S_2 \oplus P_2$. Indeed, we have $$\Ext^1_{\Gr \L}(D\L_0,D\L_0 \gsh 2)=\Ext^1_{\Gr \L}(S_1,S_2 \gsh 2) \neq 0,$$ and therefore $D\L_0$ is not graded self-orthogonal.
\end{example}

The following technical proposition provides us with a class of modules $M$ for which $G_T(M)$ is also a module and the counit map $\phi_M$ is an isomorphism. The module $D\L$ turns out to be in this class (which we later call the class of \emph{$T$-coKoszul} modules, see Theorem \ref{duamod}).

\begin{prop}\label{fincog}
Let $\L=\bigoplus_{i\geq 0} \L_i$ be a graded algebra with $\gldim \L_0 < \infty$. Suppose $\L$ is a Koszul algebra with respect to a module $T$. Let $M$ be a locally finite dimensional graded $\L$-module which is bounded above and has the property that $\Ext^i_{\Gr \L}(T \gsh{-j},M)=0$ whenever $i \neq j$. Then
\begin{itemize}
\item[(a)] $G_T(M)$ is a locally finite dimensional graded $\G$-module (stalk complex concentrated in cohomological degree $0$), and $G_T(M)_j=0$ whenever $j<0$.
\item[(b)] $M$ is a finitely cogenerated graded $\L$-module, and $M_j=0$ whenever $j>0$.
\item[(c)] The counit map $\phi_M \colon F_T G_T(M) \rightarrow M$ is an isomorphism.
\item[(d)] $G_T(M)$ is a finitely generated graded $\G$-module.
\end{itemize}
\end{prop}

\begin{proof}
(a) Since $\Ext^i_{\Gr \L}(T \gsh{-j},M)=0$ whenever $i \neq j$, it follows from Proposition \ref{rhom}(e) that $(H^i G_T(M))_j \simeq \Ext^{i+j}_{\Gr \L}(T \gsh {-j},M)=0$ whenever $i \neq 0$. Since $G_T(M)$ has nonzero cohomology only in cohomological degree $0$, we identify $G_T(M)$ with the graded module $H^0 G_T(M)$. We have $$G_T(M)_j=(H^0 G_T(M))_j \simeq \Ext^j_{\Gr \L}(T \gsh {-j},M)$$ for any $j \in \Z$, so $G_T(M)_j=0$ whenever $j<0$. Since $T$ has a locally finite dimensional graded projective $\L$-resolution, we have that $\Ext^j_{\Gr \L}(T \gsh {-j},M)$ is finite dimensional for each $j \in \Z$, so $G_T(M)$ is locally finite dimensional.

(b) Since $M$ is locally finite dimensional and bounded above, it is cogenerated by its graded socle $\soc_{\gr} M$. If $\soc_{\gr} M$ has support in degree $j$, then $\Hom_{\Gr \L}(\L_0 \gsh j,M) \neq 0$. Let $0 \to \L_0 \to T_0 \to T_1 \to \ldots \to T_n \to 0$ be a coresolution of $\L_0$ with $T_i$ in $\add T$ for all $0 \leq i \leq n$. Since $\Hom_{\Gr \L}(\L_0 \gsh j,M)=\Hom_{\mathcal D \Gr \L}(\L_0 \gsh j,M) \neq 0$, we can find an integer $0 \leq l \leq n$ such that $\Hom_{\mathcal D \Gr \L}(T_l \gsh j [-l],M) \neq 0$. Let $P$ be the graded projective $\G$-module $P=G_T(T_l)$. Since $F_T G_T(T_l) \simeq T_l$, we have
\begin{eqnarray*}
\Hom_{\mathcal D \Gr \L}(T_l \gsh j [-l],M) &\simeq& \Hom_{\mathcal D \Gr \G}(P \gsh {-j}[-j-l],G_T(M))\\ &\simeq& \Ext^{j+l}_{\Gr \G}(P \gsh {-j},G_T(M)).
\end{eqnarray*}
Since $\Hom_{\mathcal D \Gr \L}(T_l \gsh j [-l],M) \neq 0$ and $P$ is projective, we have $j+l=0$, and therefore $-n \leq j \leq 0$. So the degrees where $\soc_{\gr} M$ has support form a subset of $\{-n,\ldots,0\}$. Since $M$ is locally finite dimensional, it means that $M$ is finitely cogenerated. We also conclude that $M_j=0$ whenever $j>0$.

(c) Since $T$ is graded self-orthogonal, the map $\phi_T \colon F_T G_T (T) \to T$ is an isomorphism by Proposition \ref{rhom}(b). Since $T$ is a tilting $\L_0$-module, it follows from the Five Lemma that $\phi_{\L_0} \colon F_T G_T(\L_0) \rightarrow \L_0$ is an isomorphism. Since $\gldim \L_0 < \infty$, it  also follows from the Five Lemma that $\phi_N \colon F_T G_T(N) \rightarrow N$ is an isomorphism for any finitely generated $\L_0$-module $N$. By Proposition \ref{rhom}(d), it follows that $\phi_{N \gsh j} \colon F_T G_T(N \gsh j) \rightarrow N \gsh j$ is an isomorphism for any $j \in \Z$. Further use of the Five Lemma allows us to conclude that $\phi_N \colon F_T G_T(N) \rightarrow N$ is an isomorphism for any finite dimensional graded $\L$-module $N$.

We next prove that $F_T G_T(M)$ has nonzero cohomology only in cohomological degree $0$. For any $l>0$, Proposition \ref{rhom}(f) implies that
\begin{eqnarray*}
D(H^i F_T G_T(M))_l &\simeq& \Hom_{\mathcal D \Gr \L} (F_T G_T(M),D \L \gsh l [-i])\\ &\simeq& \Hom_{\mathcal D \Gr \G}(G_T(M),DT \gsh {-l} [-l-i])\\&=&0,
\end{eqnarray*}
since $DT$ is concentrated in degree $0$ and $G_T(M)_j=0$ whenever $j<0$.
Now suppose $l \leq 0$ is an integer, and let $U=M_{< l}$. By Proposition \ref{rhom}(e), we have $(H^i G_T(U))_j \simeq \Ext^{i+j}_{\Gr \L}(T,U \gsh j)$, which is $0$ whenever $j \leq -l$, since $T$ is concentrated in degree $0$. So
\begin{eqnarray*}
D(H^i F_T G_T(U))_l &\simeq& \Hom_{\mathcal D \Gr \L} (F_T G_T(U),D\L \gsh l [-i])\\ &\simeq& \Hom_{\mathcal D \Gr \G}(G_T(U),DT \gsh {-l} [-l-i])\\&=&0
\end{eqnarray*}
for all $i \in \Z$, since $DT$ is concentrated in degree $0$. There is a triangle $$F_T G_T(M_{\geq l}) \rightarrow F_T G_T(M) \rightarrow F_T G_T(U) \rightarrow F_T G_T(M_{\geq l})[1],$$ so the cohomology long-exact sequence gives $$(H^i F_T G_T(M_{\geq l}))_l \simeq (H^i F_T G_T(M))_l$$ for all $i \in \Z$. Since $M_{\geq l}$ is finite dimensional, the map $$\phi_{M_{\geq l}} \colon F_T G_T(M_{\geq l}) \rightarrow M_{\geq l}$$ is an isomorphism. We conclude that $(H^i F_T G_T(M))_l=0$ for all $i \neq0$ and that $(H^0 F_T G_T(M))_l \simeq M_l$.

Since $F_T G_T(M)$ has nonzero cohomology only in cohomological degree $0$, we can identify $F_T G_T(M)$ with the graded module $H^0 F_T G_T(M)$. We want to prove that $\phi_M \colon F_T G_T(M) \rightarrow M$ is an isomorphism. If $l>0$, then $(F_T G_T(M))_l=0=M_l$. Suppose $l \leq 0$. We have the following morphism of triangles.
$$\xymatrix @M=1.5ex {F_T G_T(U)[-1] \ar[r] \ar[d]^{\phi_{U[-1]}} & F_T G_T(M_{\geq l}) \ar[r] \ar[d]^{\wr} & F_T G_T(M) \ar[r] \ar[d]^{\phi_M} & F_T G_T(U) \ar[d]^{\phi_U} \\U[-1] \ar[r] &
M_{\geq l} \ar[r] & M \ar[r]& U.}$$
By taking degree $l$ parts we get the diagram
$$\xymatrix @M=1.5ex {0 \ar[r] & M_l \ar[r]^-{\sim} \ar[d]^{\wr} & (F_T G_T(M))_l \ar[r] \ar[d]^{(\phi_M)_l} & 0\\
0 \ar[r] & M_l \ar[r]^\sim & M_l \ar[r] & 0.}$$
So the map $(\phi_M)_l \colon (F_T G_T(M))_l \rightarrow M_l$ is an isomorphism for every $l \in \Z$, and therefore $\phi_M \colon F_T G_T(M) \rightarrow M$ is an isomorphism.

(d) The proof is similar to the proof of part (b). We know from part (a) that $G_T(M)$ is a locally finite dimensional graded $\G$-module that is bounded below. Therefore it is generated by its graded top $\top_{\gr} (G_T(M))=G_T(M)/J(G_T(M))$, where $J(G_T(M))$ denotes the graded Jacobson radical of $G_T(M)$. If $\top_{\gr} (G_T(M))$ has support in degree $j$, then $$\Hom_{\Gr \G}(G_T(M),D \G_0 \gsh j) \neq 0.$$ By Proposition \ref{rhom}(f), the $\G$-module $DT \simeq G_T(D\L)$ is concentrated in degree $0$. Considered as a module over $\G_0 \cong [\End_\L(T)]^{\op} \cong [\End_{\L_0}(T)]^{\op}$, the module $DT$ is a cotilting module. Let $0 \to U_m \to \ldots \to U_1 \to U_0 \to D\G_0 \to 0$ be a resolution of $D\G_0$ with $U_i$ in $\add DT$ for all $0 \leq i \leq m$. Since $\Hom_{\Gr \G}(G_T(M),D \G_0 \gsh j) = \Hom_{\mathcal D \Gr \G}(G_T(M),D \G_0 \gsh j) \neq 0$, there is an integer $0 \leq l \leq m$ such that $\Hom_{\mathcal D \Gr \G}(G_T(M),U_l \gsh j [l]) \neq 0$. Let $I$ be the graded injective $\L$-module such that $G_T(I)=U_l$. Since $F_T G_T(M) \simeq M$, we have
\begin{eqnarray*}
\Hom_{\mathcal D \Gr \G}(G_T(M),U_l \gsh j [l]) &\simeq& \Hom_{\mathcal D \Gr \L}(M,I \gsh {-j}[l-j])\\ &\simeq& \Ext^{l-j}_{\Gr \L}(M,I \gsh {-j}).
\end{eqnarray*}
Since $\Hom_{\mathcal D \Gr \G}(G_T(M),U_l \gsh j [l]) \neq 0$ and $I$ is injective, we have $l-j=0$, and therefore $0 \leq j \leq m$. So the degrees where $\top_{\gr} (G_T(M))$ has support form a subset of $\{0,\ldots,m\}$. Since $G_T(M)$ is locally finite dimensional, it means that $M$ is finitely generated.
\end{proof}

We can now prove that our new definition is equivalent to the one in \cite{Ext}, as long as $\gldim \L_0 < \infty$. As a consequence the conditions in new definition are strong enough to yield a good duality theory.

\begin{thm}\label{comp}
Let $\L=\bigoplus_{i\geq 0} \L_i$ be a graded algebra with $\gldim \L_0 < \infty$. Let $T$ be a graded $\L$-module concentrated in degree $0$. Suppose $T$ is graded self-orthogonal as a $\L$-module. Then $T$ is a tilting $\L_0$-module if and only if the counit map $\phi_{D\L} \colon F_T G_T(D \L) \rightarrow D \L$ is an isomorphism.
\end{thm}

\begin{proof}
Assume $T$ is a tilting $\L_0$-module. Since $\Ext^i_{\Gr \L}(T \gsh{-j},D\L)=0$ except when $i=j=0$, by Proposition \ref{fincog}(c) the counit map $\phi_{D\L} \colon F_T G_T(D \L) \rightarrow D \L$ is an isomorphism.

For the converse, suppose $\phi_{D\L} \colon F_T G_T(D \L) \rightarrow D \L$ is an isomorphism. From Proposition \ref{rhom}(f) we know that $G_T(D \L) \simeq \Hom_{\Gr \L}(T, D\L) \simeq DT$. Let $$\mathcal Q \colon \ldots \rightarrow Q^{-(j+1)} \rightarrow Q^{-j} \rightarrow Q^{-(j-1)}  \rightarrow \ldots \rightarrow Q^{-1} \rightarrow Q^0$$ be a graded projective $\G$-resolution of $DT$. Then the complex
\begin{multline*}
T \otimes_{\Gr \G} \mathcal Q \colon \ldots \rightarrow T \otimes_{\Gr \G} Q^{-(j+1)} \rightarrow T \otimes_{\Gr \G} Q^{-j} \rightarrow T \otimes_{\Gr \G} Q^{-(j-1)} \rightarrow \ldots\\ \ldots \rightarrow T \otimes_{\Gr \G} Q^{-1} \rightarrow T \otimes_{\Gr \G} Q^0
\end{multline*} computes the $\Tor$ groups $\Tor^{\Gr \G}_j(T, DT)$.
Since $T$ is concentrated in degree $0$, this complex is isomorphic to \begin{multline*}
T \otimes_{\G_0} \mathcal Q_0 \colon \ldots \rightarrow T \otimes_{\G_0} (Q^{-(j+1)})_0 \rightarrow T \otimes_{\G_0} (Q^{-j})_0 \rightarrow T \otimes_{\G_0} (Q^{-(j-1)})_0 \rightarrow \ldots\\ \ldots \rightarrow T \otimes_{\G_0} (Q^{-1})_0 \rightarrow T \otimes_{\G_0} (Q^0)_0,
\end{multline*} where $\mathcal Q_0$ is the complex $$\mathcal Q_0 \colon \ldots \rightarrow (Q^{-(j+1)})_0 \rightarrow (Q^{-j})_0 \rightarrow (Q^{-(j-1)})_0  \rightarrow \ldots \rightarrow (Q^{-1})_0 \rightarrow (Q^0)_0.$$ This complex is a projective $\L_0$-resolution of $DT$. We conclude that $$\Tor^{\G_0}_j(T, DT) \simeq \Tor^{\Gr \G}_j(T, DT)$$ for all $j \geq 0$.
Since $(\phi_{D\L})_0 \colon (F_T(DT))_0 \rightarrow (D\L)_0$ is an isomorphism, we have $\Tor^{\Gr \G}_j(T, DT)=0$ for all $j \neq 0$ and $(F_T(DT))_0 \simeq T \otimes_{\Gr \G} DT \simeq T \otimes_{\G_0} DT \simeq (D\L)_0$.

Consider the pair of adjoint functors $$\bar G_T=\mathbb R \Hom_{\L_0}(T,-) \colon \mathcal D^b(\mod \L_0) \to \mathcal D^b(\mod \G_0)$$ and $$\bar F_T= T \otimes^{\mathbb L}_{\G_0} - \colon \mathcal D^b(\mod \G_0) \rightarrow \mathcal D^b(\mod \L_0).$$ We have just shown that $\bar \phi_{(D\L)_0} \colon \bar F_T \bar G_T ((D\L)_0) \simeq \bar F_T (DT) \rightarrow (D\L)_0$ is an isomorphism, where $\bar \phi$ is the counit of the adjunction. Since $T$ is a self-orthogonal $\L_0$-module, it follows from \cite[Proposition 8.3]{Ext} (see also \cite[Proposition 5]{Self}) that $T$ is a Wakamatsu tilting $\L_0$-module. Since $\gldim \L_0 < \infty$, by Proposition \ref{fin} the module $T$ must be a tilting $\L_0$-module.
\end{proof}

\subsection{Koszul duality for algebras}

We next prove the duality theorem for $T$-Koszul algebras.

\begin{thm}\label{kosdua}
Let $\L=\bigoplus_{i\geq 0} \L_i$ be a graded algebra with $\gldim \L_0 < \infty$. Suppose $\L$ is a Koszul algebra with respect to a module $T$. Then \begin{itemize}
\item[(a)]$\gldim \G_0 < \infty$, and $\G$ is a Koszul algebra with respect to $_{\G}DT$.
\item[(b)] There is an isomorphism of graded algebras $$\L \cong [\bigoplus_{i\geq 0} \Ext_\G^i (DT,DT)]^{\op} \cong [\bigoplus_{i\geq 0} \Ext^i_{\Gr \G}(DT,DT \gsh i)]^{\op}.$$
\end{itemize}
\end{thm}

\begin{proof}
(a) We have $\G_0 \cong [\End_\L(T)]^{\op} \cong [\End_{\L_0}(T)]^{\op}$. Since $\gldim \L_0 < \infty$ and $T$ is a tilting $\L_0$-module, also $\gldim \G_0 < \infty$. Considered as a $\G_0$-module, the module $DT \simeq G_T(D\L)$ is a cotilting module, and by Proposition \ref{fin} it is also a tilting module. Using Proposition \ref{rhom}(e), the adjointness of $G_T$ and $F_T$, and Theorem \ref{comp}, we get
\begin{eqnarray*}
\Ext^i_{\Gr \G}(DT,DT \gsh j)&=&\Hom_{\mathcal D \Gr \G}(DT,DT\gsh j [i])\\ &\simeq &\Hom_{\mathcal D \Gr \G}(G_T(D\L),G_T(D\L)\gsh j [i])\\ &\simeq& \Hom_{\mathcal D \Gr \L}(F_T G_T(D\L),D\L \gsh {-j}[i-j])\\ &\simeq& \Hom_{\mathcal D \Gr \L}(D\L,D\L \gsh {-j}[i-j])\\ &\simeq&\Ext^{i-j}_{\Gr \L}(D\L,D\L \gsh {-j})
\end{eqnarray*}
Since $D\L$ is injective, we have $\Ext^i_{\Gr \G}(DT,DT \gsh j)=0$ whenever $i \neq j$. So $DT$ is a graded self-orthogonal $\G$-module, and $\G$ is a Koszul algebra with respect to $DT$.

(b) For each $i \geq 0$, we have
\begin{eqnarray*}
\Ext^i_{\Gr \G}(DT,DT \gsh i)&=&\Hom_{\mathcal D \Gr \G}(DT,DT\gsh i [i])\\ &\simeq& \Hom_{\mathcal D \Gr \L}(D\L,D\L \gsh {-i})\\ &\simeq&\L_i.
\end{eqnarray*}
These canonical isomorphisms define a bijective map $$\psi \colon [\bigoplus_{i\geq 0} \Ext^i_{\Gr \G}(DT,DT \gsh i)]^{\op} \rightarrow \L.$$ The graded algebra $[\bigoplus_{i\geq 0} \Ext^i_{\Gr \G}(DT,DT \gsh i)]^{\op}$ can be codified as the full subcategory $\mathcal U=\{DT \gsh i [i] \mid i \in \Z\} \subseteq \mathcal D (\Gr \G)$.  The graded algebra $\L$ can be codified as the full subcategory $\mathcal V=\{D\L \gsh {-i} \mid i \in \Z\} \subseteq \mathcal D (\Gr \L)$. Combining Theorem \ref{comp} with Proposition \ref{rhom}(d), we find that $F_T G_T(D\L \gsh {-i}) \simeq (D\L \gsh {-i})$ for all $i \in \Z$. By the adjointness of $F_T$ and $G_T$, the category $\mathcal V$ is equivalent to its essential image under $G_T$. By Proposition \ref{rhom}(c)(f) we have $G_T(D\L \gsh {-i}) \simeq DT \gsh i [i]$ for all $i \in \Z$. So $\mathcal U$ and $\mathcal V$ are equivalent as categories. Therefore $[\bigoplus_{i\geq 0} \Ext^i_{\Gr \G}(DT,DT \gsh i)]^{\op}$ and $\L$ are isomorphic as graded algebras.
\end{proof}

We call the pair $(\G,DT)$ the \emph{Koszul dual} of $(\L,T)$.

If $\L=\L_0$, then $\G=\G_0=[\End_{\L_0}(T)]^{\op}$. Similarly $[\bigoplus_{i\geq 0} \Ext_\G^i (DT,DT)]^{\op} \simeq [\End_{\G_0}(DT)]^{\op}$, so the theorem just states some basic facts in tilting theory.

If $\L_0 \cong k^{\times r}$, the theorem above is the well known duality for classical Koszul algebras \cite[Theorem 2.10.2]{Bei}.

\subsection{Koszul duality for modules}

As in the case of classical Koszul algebras, the functor $G_T =\mathbb R \Hom_{\Gr \L}(X,-) \colon \mathcal D (\Gr \L) \to \mathcal D(\Gr \G)$ is an equivalence only in special cases (when $\L$ is finite-dimensional and of finite global dimension \cite{Ext}), but restricts in different ways to equivalences between certain important pairs of (full) subcategories. We get the largest such equivalence if we set $$\mathcal T=\{M \in \mathcal D (\Gr \L) \mid \phi_M \colon F_T G_T (M) \rightarrow M \text{ is an isomorphism} \},$$ and let $\mathcal S$ be the essential image of $\mathcal T$ under $G_T$. Both $\mathcal T$ and $\mathcal S$ are closed under triangles, the translation functor, and graded shifts. By the adjointness of $G_T$ and $F_T$, there is an equivalence of triangulated categories $G_T \colon \mathcal T \rightarrow \mathcal S$.

Using Proposition \ref{fincog}, we find the largest equivalence between categories of suitably bounded locally finite-dimensional modules given by $G_T$. Define the category of \emph{$T$-coKoszul $\L$-modules} $$\cK_T(\L)=\{M \in \cgr \L \mid \Ext^i_{\Gr \L}(T \gsh{-j},M)=0 \text{ whenever } i \neq j \}$$ and the category of \emph{$T$-Koszul $\G$-modules} $$\K_T(\G)=\{N \in \gr \G \mid \Ext^i_{\Gr \G}(N,DT \gsh j)=0 \text{ whenever } i \neq j \}.$$ In the case $\L_0 \cong k^{\times r}$, these definitions are in agreement with the classical ones.

\begin{thm}\label{duamod}
The functor $G_T \colon \mathcal D (\Gr \L) \to \mathcal D(\Gr \G)$ restricts to an equivalence $E_T \colon \cK_T(\L) \rightarrow \K_T(\G)$, where $E_T=\bigoplus_{i\geq 0} \Ext^i_{\Gr \L}(T \gsh{-i},-).$
$$\xymatrix @M=1.5ex {\mathcal D(\Gr \L) \ar[rr]^{\mathbb R \Hom_{\Gr \L}(X,-)}  && \mathcal D(\Gr \G)\\
\cK_T(\L) \ar[rr]^{E_T} \ar@{^{(}->}[u] && \K_T(\G)
\ar@{^{(}->}[u]}$$
\end{thm}

\begin{proof}
If $M \in \cK_T(\L)$, then by Proposition \ref{fincog}(c) the map $\phi_M \colon F_T G_T(M) \rightarrow M$ is an isomorphism. Hence $\cK_T(\L)$ is a subcategory of $\mathcal T$, and therefore $\cK_T(\L)$ is equivalent to its essential image under $G_T$. Let $G_T(\cK_T(\L))$ denote this essential image. Since $F_T (G_T (M))$ is concentrated in cohomological degree $0$ for any $M \in \cK_T(\L)$, it also follows from Proposition \ref{fincog} that $G_T(\cK_T(\L))$ is a full subcategory of $\K_T(\G)$.

Let $N \in \K_T(\G)$. In a manner analogous to the proof of Proposition \ref{fincog}, we can show that $F_T (N) \in \cK_T(\L)$ and that the unit map $\eta_N \colon N \to G_T F_T (N)$ is an isomorphism. Therefore $G_T(\cK_T(\L))=\K_T(\G)$

As a consequence of Proposition \ref{rhom}(e), the functor $G_T$, when restricted to $\cK_T(\L)$, is isomorphic to $E_T$.
\end{proof}

In the case $\L=\L_0$, any graded $\L$-module decomposes into a direct sum of its graded parts. In this case a locally finite-dimensional module $M=\bigoplus_{i \in \Z} M_i$ is in $\cK_T(\L)$ if and only if (i) $M$ is finite-dimensional, (ii) $M_i=0$ for all $i>0$, and (iii) $\Ext^j_{\L}(T,M_i)=0$ for all $j \neq -i$. Theorem \ref{duamod} is then essentially a reformulation of \cite[Theorem 1.16]{Miya}. The category $_{\L} T^{\perp}$ coincides with the full subcategory of $\cK_T(\L)$ consisting of modules that are cogenerated in degree $0$.

In earlier papers \cite{TKos} \cite{Ext}, Koszul modules were defined to be generated in degree $0$ (and by duality coKoszul modules were defined to be cogenerated in degree $0$), so that one would have $_{\L} T^{\perp}=\cK_T(\L)$ in the $\L=\L_0$ case. It is important to note that this extra provision comes at a price. For general $T$-Koszul algebras, if all coKoszul modules $M$ are required to be cogenerated in degree $0$, then even further, more unnatural, conditions have to be put on such modules $M$ in order to make sure that $E_T(M)$ is always generated in degree $0$. The approach taken in the present paper is to make definitions as simple as possible.

In the classical Koszul case with $\L_0 \cong k^{\times r}$, it follows from the definition that Koszul modules are generated in degree $0$.

Important objects in $\cK_T(\L)$ are $T$ and the graded injective cogenerator $D\L$. Important objects in $\K_T(\G)$ are $DT$ and $\G$. At the core of the equivalence $E_T$ there is a correspondence between indecomposable direct summands of $T$ and indecomposable projective $\G$-modules and another correspondence between indecomposable graded injective $\L$-modules and indecomposable direct summands of $DT$.
\begin{eqnarray*}
&T \overset {E_T} {\longmapsto}  \G, \qquad &D\L \overset {E_T} {\longmapsto} DT,\\&T \underset {F_T} {\longmapsfrom} \G, \qquad &D\L \underset {F_T} {\longmapsfrom} DT.
\end{eqnarray*}

The equivalence $E_T$ can be composed with the duality $D$ and we get a duality $E_T D \simeq \bigoplus_{i\geq 0} \Ext^i_{\Gr \L^{\op}}(-,DT \gsh i) \colon \K_T(\L^{\op}) \rightarrow \K_T(\G)$. In this way indecomposable direct summands of $DT$ on one side correspond to indecomposable projective modules on the other side in both directions.
\begin{eqnarray*}
&DT \overset {E_T D} {\longmapsto}  \G, \qquad &\L \overset {E_T D} {\longmapsto} DT,\\&DT \underset {D F_T} {\longmapsfrom} \G, \qquad &\L \underset {D F_T} {\longmapsfrom} DT.
\end{eqnarray*}

The functor $G_T$ also restricts to equivalences in the style of \cite[Theorem 2.4]{Dua}. We consider here a bounded version of that theorem. Let $\mathcal F_{\gr \L}(T)$ be the full subcategory of $\gr \L$ consisting of modules $M$ having a finite filtration $0=M_0 \subseteq M_1 \subseteq \ldots \subseteq M_t=M$ where the factors $M_i/M_{i-1}$ are graded shifts of direct summands of $T$ for all $1 \leq i \leq t$. A cochain complex of graded modules $\ldots \rightarrow N^{-1} \rightarrow N^0 \rightarrow N^1 \rightarrow \ldots \rightarrow N^{j-1} \rightarrow N^j \rightarrow \ldots$ is called \emph{linear} if $N^j$ is generated in degree $-j$ for all $j \in \Z$. Let $\mathcal L^b(\G)$ denote the category of bounded linear cochain complexes of graded projective $\G$-modules and cochain maps. A homotopy map between linear complexes is always $0$, so $\mathcal L^b(\G)$ can be regarded as a full subcategory of $\mathcal D(\Gr \G)$.

\begin{thm}\label{delta}
The functor $G_T \colon \mathcal D (\Gr \L) \to \mathcal D(\Gr \G)$ restricts to an equivalence $G_T \colon \mathcal F_{\gr \L}(T) \rightarrow \mathcal L^b(\G)$.
$$\xymatrix @M=1.5ex {\mathcal D(\Gr \L) \ar[rr]^{\mathbb R \Hom_{\Gr \L}(X,-)}  && \mathcal D(\Gr \G)\\
\mathcal F_{\gr \L}(T) \ar[rr] \ar@{^{(}->}[u] && \mathcal L^b(\G)
\ar@{^{(}->}[u]}$$
\end{thm}

\begin{proof} The direct summands of $T$ are objects in $\mathcal T$. Since $\mathcal T$ is closed under graded shifts and triangles, the category $\mathcal F_{\gr \L}(T)$ is a subcategory of $\mathcal T$. Hence $\mathcal F_{\gr \L}(T)$ is equivalent to its essential image under $G_T$.

If $T_{(h)}$ is a direct summand of $T$, then for any $i \in \Z$ we have that $G_T(T_{(h)} \gsh i) \simeq (G_T(T_{(h)})) \gsh {-i}[-i]$ is isomorphic to an object in $\mathcal L^b(\G)$. Suppose $0 \rightarrow L \rightarrow M \rightarrow N \rightarrow 0$ is an exact sequence of objects in $\mathcal F_{\gr \L}(T)$. If $G_T(L)$ and $G_T(N)$ are isomorphic to objects in $\mathcal L^b(\G)$, then $G_T(M)$ is isomorphic to the bounded linear complex $\cone(G_T(N)[-1] \rightarrow G_T(L))$. By induction on the length of the filtration of $M$, we get that $G_T(M)$ is isomorphic to an object in $\mathcal L^b(\G)$ for any $M$ in $\mathcal F_{\gr \L}(T)$.

Let $$N \colon 0 \rightarrow N^a \rightarrow N^{a+1} \rightarrow \ldots \rightarrow N^{a+n} \rightarrow 0$$ be a bounded linear complex of graded projective $\G$-modules. The subcomplex $$N' \colon 0 \rightarrow N^{a+n} \rightarrow 0$$ is isomorphic to $G_T(T' \gsh {a+n})$ for some $T'$ in $\add T$. The quotient $$N/N' \colon 0 \rightarrow  N^a \rightarrow N^{a+1} \rightarrow \ldots \rightarrow N^{a+n-1} \rightarrow 0$$ is also a bounded linear complex. Since $G_T$ is a triangulated functor, it follows by induction on $n$ that $N \simeq G_T(M)$ for some $M$ in $\mathcal F_{\gr \L}(T)$, and therefore $G_T \colon \mathcal F_{\gr \L}(T) \rightarrow \mathcal L^b(\G)$ is an equivalence.
\end{proof}

We have the following corollary.

\begin{cor} The subcategory $\mathcal F_{\gr \L}(T)$ of $\gr \L$ is closed under direct summands.
\end{cor}

\begin{proof}
Since idempotents split in $\mathcal L^b(\G)$, the same must be true in $\mathcal F_{\gr \L}(T)$.
\end{proof}

If $\L=\L_0$, then $\Ext_\L^1(T,T)=0$, and the objects of $\mathcal F_{\gr \L}(T)$ are finite direct sums of graded shifts of direct summands of $T$.

In the classical Koszul case with $\L_0 \cong k^{\times r}$, the category $\mathcal F_{\gr \L}(T)$ consists of all finite-dimensional graded $\L$-modules.

If certain finiteness conditions are satisfied, we get an equivalence of bounded derived categories of finitely generated graded modules, as the following theorem shows. For classical Koszul algebras, this is \cite[Theorem 2.12.6]{Bei}.

\begin{thm}\label{derequiv}
Suppose $\L$ is artinian and $\G$ is noetherian. Assume $\gldim \G < \infty$. Then there is an equivalence of triangulated categories $G_T^b \colon \mathcal D^b(\gr \L) \rightarrow \mathcal D^b(\gr \G)$.
\end{thm}

\begin{proof}
If $\L$ is artinian, then $\mathcal D^b(\gr \L)$ is equivalent to the smallest triangulated subcategory of $\mathcal D (\Gr \L)$ containing all graded shifts of direct summands of $T$. If $\G$ is noetherian and of finite global dimension, then $\mathcal D^b(\gr \G)$ is equivalent to the smallest triangulated subcategory of $\mathcal D (\Gr \G)$ containing all indecomposable graded projective $\G$-modules. Since $T \in \mathcal T$ and $G_T(T)=\G$, it follows from Proposition \ref{rhom}(c) that there is an equivalence of triangulated categories $G_T^b \colon \mathcal D^b(\gr \L) \rightarrow \mathcal D^b(\gr \G)$.
\end{proof}

\subsection{What if $\L_0$ has infinite global dimension?}\label{inftysec}

So far our assumption has been that $\gldim \L_0 < \infty$, but in this subsection we remark on the situation when $\gldim \L_0=\infty$. In \cite{TKos} and \cite{Ext}, $T$-Koszul algebras were defined without any condition on the global dimension of $\L_0$. If, following \cite{Ext}, we take ``$\phi_{D\L} \colon F_T G_T(D \L) \rightarrow D \L$ is an isomorphism'' to be the defining property of $T$-Koszul algebras, can we find a more elementary definition in the style of Definition \ref{nydef} that is also appropriate in the infinite global dimension case? As stated, the important Theorem \ref{opp} and the possibly even more important Theorem \ref{comp} (both directions) are no longer true if we allow $\gldim \L_0=\infty$. Some of the theory can be recovered by using a crucial insight from \cite{TKos}, namely that one should allow $T$ to be a Wakamatsu tilting $\L_0$-module.

In the special case when $\L=\L_0$, it is known that given a self-orthogonal $\L$-module $T$, the map $\phi_{D\L} \colon F_T G_T(D \L) \rightarrow D \L$ is an isomorphism if and only if $T$ is a Wakamatsu tilting module (see for example \cite[Proposition 8.3]{Ext}). So we define $\L_0$ to be Koszul with respect to $T$ if $T$ is a Wakamatsu tilting $\L_0$-module. Can the definition for more general graded algebras be approached in a similar way? In other words, can we determine how the appropriate graded version of a Wakamatsu tilting module should be defined? If $T$ is a graded self-orthogonal $\L$-module and $\phi_{D\L}$ is an isomorphism, then it follows from the proof of Theorem \ref{comp} that $T$ is a Wakamatsu tilting $\L_0$-module, without any assumption on the global dimension of $\L_0$. In line with Definition \ref{nydef}, it seems natural to impose the following conditions on $T$.
\begin{itemize}
\item[(W1)] $T$ is a Wakamatsu tilting $\L_0$-module.
\item[(W2)] $T$ is a graded self-orthogonal $\L$-module.
\end{itemize}
As natural as these conditions might be, it is our impression that (W1) and (W2) together are not strong enough to imply that $\phi_{D\L}$ is an isomorphism, but we do not have a counterexample.

Our discussion on $T$-Koszul algebras with $\gldim \L_0=\infty$ might be related to some important unresolved homological conjectures. We mention here a possible connection with the Wakamatsu Tilting Conjecture. Let $\L=\L_0 \oplus \L_1$ be a graded algebra and $T$ a Wakamatsu tilting $\L_0$-module of finite projective dimension. Then $\phi_{D\L} \colon F_T G_T(D \L) \rightarrow D \L$ is an isomorphism if and only if $\phi_{(D \L)_{-1}} \colon F_T G_T((D \L)_{-1}) \rightarrow (D \L)_{-1}$ is an isomorphism. The Wakamatsu Tilting Conjecture can be reformulated as stating that for any $\L_0$-module $M \in T^{\perp}$, the map $\phi_M \colon F_T G_T(M) \rightarrow M$ is an isomorphism. All this suggests that the following question could and should be answered in the affirmative, but we do not have a proof.

\begin{ques}
Let $\L_0$ be a finite dimensional $k$-algebra and suppose $T$ is a Wakamatsu tilting $\L_0$-module of finite projective dimension for which the Wakamatsu Tilting Conjecture does not hold. Does there exist a graded algebra $\L=\L_0 \oplus \L_1$ such that (i) $T$ is a graded self-orthogonal $\L$-module, (ii) $(D \L)_{-1} \in T^{\perp}$ as a left $\L_0$-module, and (iii) $\phi_{D\L} \colon F_T G_T(D \L) \rightarrow D \L$ is not an isomorphism?
\end{ques}

\section{Quasi-hereditary Koszul algebras}

The definition of $T$-Koszul algebras is made as unrestrictive as possible. With the almost limitless choice of algebras $\L$ and modules $T$, there is a huge variety of different kinds of algebras satisfying the requirement (see \cite{TKos} for a wide range of examples).

Instead of taking the most general view, in this section we will have very specific examples in mind. Finite dimensional algebras arising in the representation theory of algebraic groups and Lie algebras often enjoy nice homological properties. For instance, such algebras often are Koszul and quasi-hereditary \cite{Icra}. We recall the definition of quasi-hereditary algebras below. Deeper phenomena resembling Koszul duality have been observed between algebras of this sort and extension algebras of so-called standard modules \cite{Stan}. In this section we show how the observed phenomena can be explained as instances of $T$-Koszul duality.

\subsection{Quasi-hereditary algebras with duality}

Let $\L$ be a finite dimensional $k$-algebra. Fix an ordering on a complete set of non-isomorphic simple $\L$-modules $S_1, \ldots, S_r$. For each $0 \leq i \leq r$, define the \emph{standard module} $\Delta_i$ to be the largest quotient of the projective module $P_i$ having no simple composition factors $S_j$ with $j>i$. Dually, define the \emph{costandard module} $\nabla_i$ to be the largest submodule of the injective module $I_i$ having no simple composition factors $S_j$ with $j>i$. Let $$\Delta=\bigoplus_{i=1}^r \Delta_i$$ and $\nabla=\bigoplus_{i=1}^r \nabla_i$.

We say that $\L$ is a \emph{quasi-hereditary algebra} \cite{Hwc} \cite{QuaH} if (i) $\Lambda$ admits a $\Delta$-filtration, i.e., there is a filtration $0=M_0 \subseteq M_1 \subseteq \ldots \subseteq M_t=\L$ where the subfactors $M_j/M_{j-1}$ are standard modules for all $1 \leq j \leq t$, and (ii) $\End_\L (\Delta_i)$ is a division ring for all $1 \leq i \leq r$. All quasi-hereditary algebras have finite global dimension.

If $\L$ is quasi-hereditary, then for all $1 \leq i,j \leq r$ and $n>0$ we have $\Ext^n_\L(\Delta_i,\nabla_j)=0$ \cite[Theorem 1.8]{Icra}. Furthermore $\Hom_\L(\Delta_i,\nabla_j)=0$ whenever $i \neq j$, and $\Hom_\L(\Delta_i,\nabla_i) \simeq \End(S_i)$.

A quasi-hereditary structure on an algebra determines a unique \emph{characteristic tilting module} $\mathfrak T_c$, a basic tilting-cotilting module that admits both a $\Delta$-filtration and a $\nabla$-filtration \cite{Tilt}. For each $1 \leq i \leq r$, the module $\mathfrak T_c$ has an indecomposable direct summand $T_i$ which admits morphisms $\Delta_i \hookrightarrow T_i \twoheadrightarrow \nabla_i$ with non-zero composition. These are the only indecomposable direct summands of $\mathfrak T_c$, so $\mathfrak T_c \simeq \bigoplus_{i=1}^r T_i$. The module $\Delta$ has a \emph{tilting coresolution} $$0 \to \Delta \to T^0 \to T^1 \to \ldots \to T^n \to 0$$ with $T^i \in \add \mathfrak T_c$ for all $0 \leq i \leq n$ \cite{Tilt}. Dually, the module $\nabla$ has a \emph{tilting resolution} $0 \to T^m \to \ldots \to T^1 \to T^0 \to \nabla \to 0$ with $T^i \in \add \mathfrak T_c$ for all $0 \leq i \leq m$.

In this section we only consider \emph{quasi-hereditary algebras with duality}, that is, we suppose our quasi-hereditary algebras are equipped with a duality functor $(-)^\circ \colon \mod \L \to \mod \L$ such that ${S_i}^\circ \simeq S_i$ for all $1 \leq i \leq r$. The direct summands of the characteristic tilting module have the property that ${T_i}^\circ \simeq T_i$ for all $1 \leq i \leq r$. Quasi-hereditary algebras with duality are called \emph{BGG algebras} in \cite{Bgg}. For such algebras the BGG reciprocity principle holds \cite{Bgg}; we have $$(P_i \colon \Delta_j)=[\Delta_j \colon S_i]$$ for all $1 \leq i ,j \leq r$, where $(P_i \colon \Delta_j)$ denotes the filtration multiplicity of $\Delta_j$ in a $\Delta$-filtration of $P_i$ and $[\Delta_j \colon S_i]$ denotes the multiplicity of $S_i$ in a composition series for $\Delta_j$.

Our interest in this section lies in graded finite dimensional algebras $\L=\bigoplus_{i=1}^t \L_i$ that are simultaneously quasi-hereditary and (classically) Koszul. It turns out that such algebras are in many cases Koszul with respect to $\Delta$. We begin with an example.

\begin{example}\label{schur}
Let $\L$ be the path algebra  $\L=kQ/I$, where $Q$ is the quiver
$$\xymatrix{1 \ar@/^/[r]^\alpha & 2 \ar@/^/[l]^{\alpha^\circ} \ar@/^/[r]^\beta & 3 \ar@/^/[l]^{\beta^\circ} \ar@/^/[r]^\gamma & 4 \ar@/^/[l]^{\gamma \circ} \ar@/^/[r]^\delta & 5 \ar@/^/[l]^{\delta^\circ}}$$
and $I=\langle \rho \rangle$ is the ideal generated by the set of relations $$\langle \rho \rangle= \{\beta \alpha, \gamma \beta, \delta \gamma, \beta^\circ \beta-\alpha \alpha^\circ, \gamma^\circ \gamma-\beta \beta^\circ, \delta^\circ \delta-\gamma \gamma^\circ, \delta \delta^\circ, \alpha^\circ \beta^\circ, \beta^\circ \gamma^\circ, \gamma^\circ \delta^\circ \}.$$ This is a quasi-hereditary algebra with duality. It is (Morita equivalent to) a block of a Schur algebra, and it is denoted by $\L_{5,0}$ in \cite{Rar} and by $\mathcal A_5$ in \cite{Schur}. The indecomposable standard modules are $\Delta_1=S_1$,
$$\begin{array}{ccccccccccc}
\Delta_2 \colon & \xymatrix@!=2pt{& \mathtt{S_2} \ar@{-}[dl]\\
\mathtt{S_1}}, && \Delta_3 \colon & \xymatrix@!=2pt{& \mathtt{S_3}
\ar@{-}[dl]\\
\mathtt{S_2}}, && \Delta_4 \colon & \xymatrix@!=2pt{& \mathtt{S_4}
\ar@{-}[dl]\\
\mathtt{S_3}}, && \Delta_5 \colon & \xymatrix@!=2pt{& \mathtt{S_5}
\ar@{-}[dl]\\
\mathtt{S_4}}. \end{array}$$
The characteristic tilting module $\mathfrak T_c=T_1 \oplus T_2 \oplus T_3 \oplus T_4 \oplus T_5$ has direct summands $T_1=S_1$,
$$\begin{array}{ccccccccccc}
T_2 \colon & \xymatrix@!=2pt{\mathtt{S_1} \ar@{-}[dr] &\\ & \mathtt{S_2}
\ar@{-}[dl]\\\mathtt{S_1}}, &
T_3 \colon & \xymatrix@!=2pt{& \mathtt{S_2} \ar@{-}[dl] \ar@{-}[dr] &
\\\mathtt{S_1} \ar@{-}[dr]&& \mathtt{S_3} \ar@{-}[dl]\\& \mathtt{S_2}&}, &
T_4 \colon & \xymatrix@!=2pt{& \mathtt{S_3} \ar@{-}[dl] \ar@{-}[dr] &
\\\mathtt{S_2} \ar@{-}[dr]&& \mathtt{S_4} \ar@{-}[dl]\\& \mathtt{S_3}&}, &
T_5 \colon & \xymatrix@!=2pt{& \mathtt{S_4} \ar@{-}[dl] \ar@{-}[dr] &
\\\mathtt{S_3} \ar@{-}[dr]&& \mathtt{S_5} \ar@{-}[dl]\\& \mathtt{S_4}&}.
\end{array}$$

The algebra is $T$-Koszul in three different important ways using three different sets of orthogonal modules: namely, the indecomposable summands of the characteristic tilting module, the simple modules, and the standard modules. The grading imposed on $\L$ is different in each case.

If all arrows are assigned degree $0$, then $\L=\L_0$ and $\L$ is Koszul with respect to $\mathfrak T_c$. The Koszul dual algebra of $(\L,\mathfrak T_c)$ is the \emph{Ringel dual} quasi-hereditary algebra $[\End_\L(\mathfrak T_c)]^{\op}$.
In this example $\L$ is \emph{Ringel self-dual}, that is, there is an isomorphism $[\End_\L(\mathfrak T_c)]^{\op} \cong \L$.

If all arrows are assigned degree $1$, then $\L_0 \cong k^{\times r}$ and $\L$ is a Koszul algebra in the classical sense. The Koszul dual algebra $[\bigoplus_{i\geq 0} \Ext_\L^i (\L_0,\L_0)]^{\op}$ is isomorphic to $kQ/I'$, where $I'=\langle \rho ' \rangle$ is the ideal generated by the set of relations $$\langle \rho ' \rangle= \{\beta^\circ \beta-\alpha \alpha^\circ, \gamma^\circ \gamma-\beta \beta^\circ, \delta^\circ \delta-\gamma \gamma^\circ, \delta \delta^\circ \}.$$

Finally, we want to show that $\L$ is Koszul with respect to $\Delta$. For this we need to do a little trick with the grading. Let $\deg \alpha = \deg \beta = \deg \gamma = \deg \delta = 1$ and $\deg \alpha^\circ = \deg \beta^\circ = \deg \gamma^\circ = \deg \delta^\circ = 0$. Then $\L$ is Koszul with respect to $\L_0 \simeq \Delta=\Delta_1 \oplus \Delta_2 \oplus \Delta_3 \oplus \Delta_4 \oplus \Delta_5$, since $\Ext^i_{\Gr \L}(\Delta,\Delta \gsh j) \neq 0$ implies $i=j$. The Koszul dual algebra $\G=[\bigoplus_{i\geq 0} \Ext_\L^i (\Delta,\Delta)]^{\op}$ is isomorphic to the graded algebra $k \check Q/\check I$, where $\check Q$ is the quiver $$\xymatrix{1 & 2 \ar@/_/[l]_{\check \alpha} \ar@/^/[l]^{\alpha^\circ} & 3 \ar@/_/[l]_{\check \beta} \ar@/^/[l]^{\beta^\circ} & 4 \ar@/_/[l]_{\check \gamma} \ar@/^/[l]^{\gamma \circ} & 5, \ar@/_/[l]_{\check \delta} \ar@/^/[l]^{\delta^\circ}}$$ the grading is given by $\deg \check \alpha = \deg \check \beta = \deg \check \gamma = \deg \check \delta = 1$ and $\deg \alpha^\circ = \deg \beta^\circ = \deg \gamma^\circ = \deg \delta^\circ = 0$, and $\check I=\langle \check \rho \rangle$ is the ideal generated by the set of relations $$\langle \check \rho \rangle= \{\alpha^\circ \check \beta - \check \alpha \beta^\circ, \beta^\circ \check \gamma - \check \beta \gamma^\circ, \gamma^\circ \check \delta - \check \gamma \delta^\circ, \alpha^\circ \beta^\circ, \beta^\circ \gamma^\circ, \gamma^\circ \delta^\circ \}.$$ From Theorem \ref{kosdua} we know there is an isomorphism of graded algebras $$\L \cong [\bigoplus_{i\geq 0} \Ext_\G^i (D \Delta,D\Delta)]^{\op}.$$ According to Theorem \ref{derequiv} there is an equivalence of triangulated categories $$\mathcal D^b(\gr \L) \to \mathcal D^b(\gr \G).$$ It restricts to the equivalence $$\mathcal F_{\gr \L}(\Delta) \to \mathcal L^b(\G)$$ from Theorem \ref{delta}.
\end{example}

\subsection{Height functions}

The algebra in Example \ref{schur} was Koszul with respect to $\Delta$, but not all graded quasi-hereditary algebras with duality are Koszul with respect to $\Delta$. Sufficient conditions can be given in terms of \emph{height functions}. A height function is a function $h\colon \{1,\ldots, r\} \rightarrow \mathbb N_{\geq 0}$.

Let $\Lambda=\bigoplus_{i=0}^t \L_{(i)}$ be a graded quasi-hereditary algebra with duality. We put degree indices in parentheses to distinguish the grading from the $\Delta$-grading below. Assume $\L_{(0)} \cong k^{\times r}$. Assume further that $\L$ is generated in degree $1$. For all $1 \leq j \leq r$, the modules $\Delta_j$, $\nabla_j$ and $T_j$ have graded lifts which preserve the morphisms $\Delta_j \hookrightarrow T_j \twoheadrightarrow \nabla_j$, see for example \cite{Stan}. The functor $(-)^{\circ}$ lifts to a duality $(-)^{\circ} \colon \gr (\L) \to \gr (\L)$ with $S_j \gsh{\gsh l}^\circ \simeq S_j \gsh{\gsh {-l}}$ for all $l \in \Z$, where $\gsh{\gsh l}$ denotes the $l$th graded shift for this grading.

Fix a height function $h$. We consider the following conditions on $h$.
\begin{itemize}
\item[(H1)] $[(\Delta_j)_{(l)} \colon S_i]=0$ whenever $h(i) \neq h(j)-l$.
\item[(H2)] For each $1 \leq j \leq r$, the minimal graded tilting coresolution of $\Delta_j$, $$0 \to \Delta_j \rightarrow T^{0,j} \rightarrow T^{1,j} \to \ldots \to T^{n,j} \to 0,$$ has the property that for all $0 \leq u \leq n$, the indecomposable direct summands of $T^{u,j}$ are of the form $T_i \gsh {\gsh {-u}}$ with $h(i)=h(j)-u$.
\end{itemize}

Condition (H2) is the same as condition (I) in \cite{Stan}. Via the duality $(-)^{\circ}$ it is also equivalent to condition (II) in that paper.

In Example \ref{schur}, the height function $h(i)=i$ for all $1 \leq i \leq 5$ satisfies conditions (H1) and (H2).

We have not assumed that $\L$ is Koszul, but this turns out to be a consequence of (H2).

\begin{prop}
If (H2) holds, then $\L$ is Koszul (in the classical sense).
\end{prop}

\begin{proof}
From (H2) and the dual statement under $(-)^\circ$ it follows that the algebra $\L$ is \emph{balanced} \cite{Maz}. Balanced quasi-hereditary algebras are Koszul \cite[Corollary 6]{Maz}.
\end{proof}

The existence of a height function satisfying (H1) and (H2) is more important than its actual values. The following lemma explains how values are related within the same block of the algebra.

\begin{lem}\label{h1}
If (H1) holds, then $e_i \L_{(1)} e_j \neq 0$ implies $|h(i) - h(j)|=1$. More precisely, suppose (H1) holds and $e_i \L_{(1)} e_j \neq 0$. Then
\begin{itemize}
\item[(a)] $h(i) - h(j)=1$ if and only if $j<i$.
\item[(b)] $h(j) - h(i)=1$ if and only if $i<j$.
\end{itemize}
\end{lem}

\begin{proof}
Suppose (H1) holds and $e_i \L_1 e_j \neq 0$. If $i<j$, then $[(\Delta_j)_{(1)} \colon S_i] \neq 0$ and $h(i)=h(j)-1$. If $i>j$, by using the duality $(-)^{\circ}$ we get that $e_j \L_1 e_i \neq 0$, and by repeating the argument we get $h(j)=h(i)-1$. We cannot have $i=j$, since $\End_\L(\Delta_i)$ is a division ring.
\end{proof}

In order to obtain Koszulity with respect to $\Delta$, we follow the same strategy as in Example \ref{schur}, and regrade our algebra $\L$. Given a height function $h$ satisfying (H1), we define what we call the \emph{$\Delta$-grading} on $\L$ in the following way. We assign degrees to the spaces $e_i \L_{(l)} e_j$, $l \geq 0$, $1 \leq i,j \leq r$ whenever $e_i \L_{(l)} e_j \neq 0$. The idempotents should be in degree $0$, so we put $\deg_{\Delta}(e_i \L_{(0)} e_i)=0$ for all $1 \leq i \leq r$. Obviously, we have $e_i \L_{(0)} e_j=0$ whenever $i \neq j$. By assumption the algebra $\L$ is generated by $\L_{(1)}$. With Lemma \ref{h1} in mind, we define the grading on generators by
$$\deg_{\Delta}(e_i \L_{(1)} e_j)=\left \{
\begin{array}{c c}
1 & \text { if } h(i)-h(j)=1\\
0 & \text { if } h(j)-h(i)=1
\end{array}
\right .$$
for all $1 \leq i,j \leq r$.
This can most conveniently be expressed as $$\deg_{\Delta}(e_i \L_{(1)} e_j)=\frac {1+h(i)-h(j)} 2.$$
Hence, for all $l \geq 0$, $1 \leq i,j \leq r$ such that $e_i \L_{(l)} e_j \neq 0$, we get the formula $$\deg_{\Delta}(e_i \L_{(l)} e_j)=\frac {l+h(i)-h(j)} 2 \geq 0.$$ The dependence of the $\Delta$-grading on the height function is only apparent, as the next lemma shows.

\begin{lem}
As long as (H1) holds, the $\Delta$-grading does not depend on the particular choice of height function $h$.
\end{lem}

\begin{proof}
For simplicity assume that $\L$ is indecomposable as an algebra. Let $h$ and $h'$ be two height functions satisfying (H1), and let $d=h(1)-h'(1)$. By repeated use of Lemma \ref{h1}, we get that $d=h(i)-h'(i)$ for all $1\leq i \leq r$. Then $\frac {l+h(i)-h(j)} 2= \frac {l+h'(i)-h'(j)} 2$ for all $l \geq 0$, $1 \leq i,j \leq r$. The general case follows.
\end{proof}

The degree zero part of $\L$ under the $\Delta$-grading is described in the following proposition.

\begin{prop}\label{null}
If (H1) holds and $\L$ is given the $\Delta$-grading, then $\L_0 \simeq \Delta$ as graded $\L$-modules.
\end{prop}

\begin{proof}
It is sufficient to prove that $\L_0 e_j \simeq \Delta_j$ for any given $1 \leq j \leq r$. Both $\L_0 e_j$ and $\Delta_j$ are quotients of $P_j=\L e_j$. Let $K_j$ denote the kernel of $P_j \twoheadrightarrow \Delta_j$. Consider an element $e_i \lambda_l e_j \in (\L_{\geq 1}) e_j$ and suppose $\lambda_l \in \L_{(l)}$. Then $h(i)>h(j)-l$, so $e_i \lambda_l e_j \in K_j$ by (H1). So $\Delta_j$ is a quotient of $\L_0 e_j$.

Let $e_i \lambda_l e_j \in e_i \L_{(l)} e_j$ with $l \geq 1$. If $l=1$ and $e_i \lambda_1 e_j \neq 0$, then $\deg_{\Delta}(e_i \lambda_1 e_j)=0$ implies $i<j$. Since the generators of $\L$ are in $\L_{(1)}$, by induction on $l \geq 1$ it follows that $\deg_{\Delta}(e_i \lambda_l e_j)=0$ implies $i<j$. The module $K_j$ is generated by elements of the form $e_i \lambda_l e_j$ with $\lambda_l \in \L_{(l)}$, $l \geq 1$ and $i>j$. Whenever $i>j$, we have $e_i \lambda_l e_j \in (\L_{\geq 1}) e_j$. Therefore $K_j=(\L_{\geq 1}) e_j$ and $\L_0 e_j \simeq \Delta_j$.
\end{proof}

As a consequence we have established that $\Delta$ is a tilting $\L_0$-module, which is one of the conditions for $\L$ being Koszul with respect to $\Delta$. The remaining crucial step is the vanishing result in the following proposition. When (H1) holds, the minimal graded tilting coresolution of $\Delta$ in the ordinary grading,
$$0 \to \Delta \to T^0 \to T^1 \to \ldots \to T^n \to 0,$$
can be regraded according to the rule $$\deg_\Delta (e_iT^{u,j}_{(l)})=\frac{l+h(i)-h(j)} 2,$$ and we obtain the minimal graded tilting coresolution of $\L_0 \simeq \Delta$ in the $\Delta$-grading.

\begin{prop}\label{vanish}
If (H1) and (H2) hold and $\L$ is given the $\Delta$-grading, then $\Hom_{\Gr \L} (\Delta, T^u \gsh v)=0$ whenever $u \neq v$.
\end{prop}

\begin{proof}
Assume $\Hom_{\Gr \L} (\Delta, T^u \gsh v) \neq 0$. Choose $i$ and $j$ such that $$\Hom_{\Gr \L} (\Delta_i, T^{u,j} \gsh v) \neq 0.$$ By (H2), the indecomposable direct summands of $T^{u,j}$ are of the form $T_s \gsh {\gsh{-u}}$ with $h(s)=h(j)-u$. An ungraded morphism from $\Delta_i$ to $T_s$ has to factor through a subfactor of $T_s$ of the form $\nabla_i$. According to \cite[Proposition 3.1.(ii)]{Stan}, the simple socle of $\nabla_i$ as a subfactor of $T_s$ is positioned in (ordinary) degree $h(s)-h(i)$. As a subfactor of $T_s \gsh {\gsh{-u}}$ it is positioned in degree $l=h(s)-h(i)-u$. Since $h(s)=h(j)-u$, we get $$\deg_\Delta (\soc \nabla_i)=\frac{l+h(i)-h(j)} 2=-u.$$ Hence, if $\Hom_{\Gr \L} (\Delta_i, T_s \gsh {\gsh{-u}} \gsh v) \neq 0$, then $u=v$, which proves the proposition.
\end{proof}

A quasi-hereditary algebra with duality is called \emph{multiplicity free} if $(P_i \colon \Delta_j)=[\Delta_j \colon S_i] \leq 1$ for all $1 \leq i,j \leq r$. The following example shows that (H1) and (H2) can hold even if $\L$ is not multiplicity free.

\begin{example}
Let $\L$ be the path algebra  $\L=kQ/I$, where $Q$ is the quiver
$$
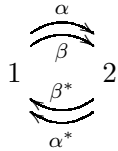
\xymatrix{1 \ar@/^0.8pc/@<1ex>[r]_\beta \ar@/^0.8pc/@<2ex>[r]^\alpha&
2  \ar@/^0.8pc/@<1ex>[l]_{\beta^\ast} \ar@/^0.8pc/@<2ex>[l]^{\alpha^\ast}}$$
and $I=\langle \rho \rangle$ is the ideal generated by the set of relations $\langle \rho \rangle=\{\alpha \alpha^\ast, \beta \alpha^\ast,\alpha \beta^\ast, \beta \beta^\ast \}$. The indecomposable standard modules are
$$\begin{array}{ccccc}\Delta_1 \colon & \xymatrix{\mathtt{S_1}}, &&
\Delta_2 \colon  & \xymatrix@!=2pt{& \mathtt{S_2} \ar@{-}[dl] \ar@{-}[dr] &\\\mathtt{S_1} && \mathtt{S_1}}.\end{array}$$
The characteristic tilting module $\mathfrak T_c=T_1 \oplus T_2$ has direct summands
$$\begin{array}{ccccc}T_1 \colon & \xymatrix{\mathtt{S_1}}, &&
T_2 \colon  & \xymatrix@!=2pt{\mathtt{S_1} \ar@{-}[dr] && \mathtt{S_1} \ar@{-}[dl]\\ & \mathtt{S_2} \ar@{-}[dl] \ar@{-}[dr] &\\\mathtt{S_1} && \mathtt{S_1}}.\end{array}$$

Define a height function by $h(1)=1$ and $h(2)=2$. Then $h$ satisfies (H1). The standard modules have graded tilting coresolutions $0 \to \Delta_1 \to T_1 \to 0$ and $$0 \to \Delta_2 \to T_2 \to T_1 \gsh{\gsh {-1}} \oplus T_1 \gsh{\gsh {-1}} \to 0,$$ so condition (H2) is also satisfied.
\end{example}

\subsection{Koszulity with respect to $\Delta$}

\begin{thm}\label{hoved}
Let $\L$ be a graded quasi-hereditary algebra with duality. Let $h$ be a height function satisfying (H1) and (H2). Regrade $\L$ according to the $\Delta$-grading. Then $\L$ is a Koszul algebra with respect to $\L_0 \simeq \Delta$.
\end{thm}

\begin{proof}
We have $\gldim \L_0 \leq \gldim \L < \infty$, and $\L_0 \simeq \Delta$ is a tilting $\L_0$-module. We only need to show that $\Ext^i_{\Gr \L}(\Delta,\Delta \gsh j)=0$ whenever $i \neq j$.

Let $0 \to \Delta \rightarrow T^0 \xrightarrow {f_0} T^1 \xrightarrow {f_1} T^2 \to \ldots \to T^n \to 0$ be a minimal graded tilting coresolution of $\Delta$. Let $K^i=\ker f_i$ for $i \geq 0$. For any $i>0$ and $j \in \Z$, since $\Ext^i_{\Gr \L}(\Delta,T \gsh j)=0$, we get from long-exact sequences the dimension shift formula $$\Ext^i_{\Gr \L}(\Delta,\Delta \gsh j) \simeq \Ext^1_{\Gr \L}(\Delta,K^{i-1} \gsh j).$$ There are also surjections $$\Hom_{\Gr \L}(\Delta,K^i \gsh j) \to \Ext^1_{\Gr \L}(\Delta,K^{i-1} \gsh j) \to 0.$$ Since $K^i$ is a submodule of $T^i$ and $\Hom_{\Gr \L} (\Delta, T^i \gsh j)=0$ whenever $i \neq j$ by Proposition \ref{vanish}, we have $\Hom_{\Gr \L} (\Delta, K^i \gsh j)=0$ whenever $i \neq j$. So $\Ext^i_{\Gr \L}(\Delta,\Delta \gsh j) \simeq \Ext^1_{\Gr \L}(\Delta,K^{i-1} \gsh j)=0$ whenever $i \neq j$.
\end{proof}

If $\L$ and $h$ are as in the above theorem, then in particular $\Delta$ is a graded self-orthogonal $\L$-module. As usual, let $\G=[\bigoplus_{i \geq 0} \Ext_{\L}^i(\Delta,\Delta)]^{\op}$. We construct the functor $G_{\Delta} \colon \mathcal D\Gr \L \to \mathcal D\Gr \G$ in the usual way described in section \ref{lift}. The indecomposable graded projective $\G$-modules concentrated in degree zero are of the form $G_{\Delta}(\Delta_i)$, $1 \leq i \leq r$.

\begin{prop}
Let $\L$ and $h$ be as in Theorem \ref{hoved}. For any $1 \leq i \leq r$,
$$G_{\Delta}(\nabla_i) \simeq {}_\G S_i,$$
where ${}_\G S_i$ is the simple top of the indecomposable graded projective $\G$-module $G_{\Delta}(\Delta_i)$.
\end{prop}

\begin{proof}
We have $(H^l G_{\Delta}(\nabla_i))_j \simeq \Ext^{l+j}_{\Gr \L}(\Delta,\nabla_i \gsh j)=0$ whenever $j \neq 0$ or $l \neq 0$. So $G_{\Delta}(\nabla_i) \simeq (H^0 G_{\Delta}(\nabla_i))_0 \simeq \Hom_{\Gr \L}(\Delta,\nabla_i) \simeq \Hom_{\Gr \L}(\Delta_i,\nabla_i)$, which is $1$-dimensional as a $k$-vector space. If $i \neq s$, then $$\Hom_{\mathcal D\Gr \G}(G_{\Delta}(\Delta_s),G_{\Delta}(\nabla_i)) \simeq \Hom_{\mathcal D\Gr \L}(\Delta_s,\nabla_i)=0,$$ so we must have $G_{\Delta}(\nabla_i) \simeq \top G_{\Delta}(\Delta_i)$.
\end{proof}

When we know that $\L$ is Koszul with respect to $\Delta$, we can apply all the $T$-Koszul machinery developed earlier in this paper and get a surprisingly clear statement of the duality theory. According to Theorem \ref{kosdua}, the $\Delta$-grading and the $\Ext$-grading coincide, so the results can be stated without reference to the $\Delta$-grading.

\begin{cor}\label{gqh}
Let $\L$ be a graded quasi-hereditary algebra with duality admitting a height function $h$ satisfying (H1) and (H2). Let $$\G=[\bigoplus_{i \geq 0} \Ext_{\L}^i(\Delta,\Delta)]^{\op}.$$ Then $$\L \cong [\bigoplus_{i \geq 0} \Ext_{\G}^i(D\Delta,D\Delta)]^{\op}$$ as ungraded algebras. Furthermore, if $\L$ and $\G$ are given the $\Ext$-grading, then there is an equivalence of triangulated categories $\mathcal D^b(\gr \L) \to \mathcal D^b(\gr \G)$ which restricts to an equivalence $\mathcal F_{\gr \L}(\Delta) \to \mathcal L^b(\G).$
\end{cor}

\begin{proof}
According to Theorem \ref{hoved}, the algebra $\L$ with the $\Delta$-grading is Koszul with respect to $\Delta$. Let $S_u$ and $S_v$ be two simple $\G$-modules. Since $\dim_k \L < \infty$ and $\gldim \L < \infty$, the groups
\begin{eqnarray*}
\Ext_{\Gr \G}^i(S_u,S_v \gsh j)&=&\Hom_{\mathcal D \Gr \G}(S_u,S_v \gsh j[i])\\
&\simeq& \Hom_{\mathcal D \Gr \L}(\nabla_u,\nabla_v \gsh {-j}[i-j])
\end{eqnarray*}
are non-zero only for finitely many values of $i$ and $j$. So $\gldim \G < \infty$. Since $\Delta \in \mathcal D^b(\gr \L)$, the category $\mathcal F_{\gr \L}(\Delta)$ is a subcategory of $\mathcal D^b(\gr \L)$. The rest follows from Theorems \ref{kosdua}, \ref{derequiv} and \ref{delta}.
\end{proof}

This corollary should be compared with Theorem 4.1 in \cite{Stan}. We believe our treatment of Koszul duality is more uniform and straightforward. It should be noted however that we assume the presence of a duality functor$(-)^\circ$, an assumption not made in \cite{Stan}. When $\L$ is a graded quasi-hereditary algebra with duality, then the conditions (I)--(IV) in \cite{Stan} imply (H1) and (H2).

It is worth writing out the statement for an important class of algebras considered in \cite{Stan}, namely the algebras with module categories equivalent to blocks of the BGG category $\mathcal O$. Category $\mathcal O$ is the category of certain representations of a semi-simple Lie algebra (\cite{Hum} is a textbook reference). The algebras arising in this way are known to be quasi-hereditary with duality. In this context the standard modules are the same as the so-called Verma modules. The height function can be defined with the help of the Weyl group. If such an algebra is multiplicity free, the conditions (H1) and (H2) are satisfied \cite[proof of Theorem 5.1]{Stan}.

\begin{cor}\label{bgg}
Let $A_\lambda$ be a basic algebra corresponding to a block $\mathcal O_\lambda$ of category $\mathcal O$ for a semi-simple Lie algebra.  Let $\Delta$ denote the direct sum of all Verma modules in $\mathcal O_\lambda$. Let $E_\lambda=[\bigoplus_{i \geq 0} \Ext_{A_\lambda}^i(\Delta,\Delta)]^{\op}$. Suppose $A_\lambda$ is multiplicity free. Then $A_\lambda \cong [\bigoplus_{i \geq 0} \Ext_{E_\lambda}^i(D \Delta,D \Delta)]^{\op}$ as ungraded algebras. Furthermore, if $A_\lambda$ and $E_\lambda$ are given the $\Ext$-grading, then there is an equivalence of triangulated categories $\mathcal D^b (\gr A_\lambda) \rightarrow \mathcal D^b (\gr E_\lambda)$ which restricts to an equivalence $\mathcal F_{\gr A_\lambda}(\Delta) \to \mathcal L^b(E_\lambda)$.
\end{cor}

\end{document}